\documentclass[a4paper,11pt]{article}
\usepackage[top=2cm, bottom=4.5cm, left=2.5cm, right=2.5cm]{geometry}
\usepackage{amsmath}
\usepackage{amsthm}
\usepackage{amssymb}
\usepackage{mathtools}
\usepackage{framed}
\usepackage[inline]{enumitem}
\usepackage{tikz}
\usepackage{color}
\usepackage{xfrac}
\usepackage[all]{xy}
\usepackage[utf8]{inputenc}
\usepackage{setspace}
\usepackage{currfile}
\usepackage{comment}
\usepackage{textcomp}
\usepackage{mathrsfs}
\usepackage[unicode]{hyperref}
\usepackage{cleveref}
\usepackage{thmtools}
\usepackage{booktabs}
\usepackage{newunicodechar}

\DeclareMathOperator{\Id}{Id}
\DeclareMathOperator{\supp}{supp}

\newcommand{\coleq}{\mathrel{\mathop:}=}

\declaretheorem[name=Theorem,numberwithin=section,refname={theorem,theorems},Refname={Theorem,Theorems}]{theorem}
\declaretheorem[name=Proposition,refname={proposition,propositions},Refname={Proposition,Propositions},sibling=theorem]{proposition}
\declaretheorem[name=Lemma,refname={lemma,lemmas},Refname={Lemma,Lemmas},sibling=theorem]{lemma}
\declaretheorem[name=Remark,refname={remark,remarks},Refname={Remark,Remarks},sibling=theorem]{remark}
\declaretheorem[name=Remarks,refname={remark,remarks},Refname={Remark,Remarks},sibling=theorem]{remarks}

\crefname{proposition}{proposition}{propositions}

\bibliography{dkt2}

\newcommand{\cB}{{\mathscr{B}}}
\newcommand{\bZ}{{\mathbb{Z}}}
\newcommand{\cC}{\mathscr{C}}
\newcommand{\cD}{\mathscr{D}}
\newcommand{\cE}{\mathscr{E}}
\newcommand{\cF}{\mathscr{F}}
\newcommand{\cK}{\mathscr{K}}

\newcommand{\cL}{\mathscr{L}}
\newcommand{\cM}{\mathscr{M}}
\newcommand{\cO}{\mathscr{O}}
\newcommand{\cS}{\mathscr{S}}
\newcommand{\bN}{\mathbb{N}}
\newcommand{\bR}{\mathbb{R}}
\newcommand{\bC}{\mathbb{C}}
\newcommand{\pd}{\partial}
\newcommand{\ud}{\mathrm{d}}

\newcommand{\abso}[1]{\left|#1\right|}
\newcommand{\norm}[1]{\left\lVert#1\right\rVert}

\title{Projective descriptions of spaces of functions and distributions}
\author{Christian Bargetz\footnote{Institut für Mathematik, Universität Innsbruck, Technikerstraße 13, 6020 Innsbruck, Austria. e-mail: \href{mailto:christian.bargetz@uibk.ac.at}{christian.bargetz@uibk.ac.at}.}, Eduard A.~Nigsch\footnote{Institute of Analysis and Scientific Computing, TU Wien, Wiedner Hauptstraße 8-10, 1040 Vienna, Austria. e-mail: \href{mailto:eduard.nigsch@tuwien.ac.at}{eduard.nigsch@tuwien.ac.at}.}, Norbert Ortner\footnote{Institut für Mathematik, Universität Innsbruck, Technikerstraße 13, 6020 Innsbruck, Austria. e-mail: \href{mailto:mathematik1@uibk.ac.at}{mathematik1@uibk.ac.at}.}}

\begin{document}
\maketitle

\begin{abstract}
We present projective descriptions of classical spaces of functions and distributions. More precisely, we provide descriptions of these spaces by semi-norms which are defined by a combination of classical norms and multiplication or convolution with certain functions. These seminorms are simpler than the ones given by a supremum over bounded or compact sets.
\end{abstract}



\section{Introduction and Notation}\label{sec1}

Spaces of distributions typically are dual spaces of spaces of smooth functions and carry the strong topology which is the one of uniform convergence on bounded sets. When working with these spaces such a description might be unhandy and a description with seminorms in terms of convolution and multiplication with smooth functions and classical norms might be preferable. We use the notion of a distribution space in the sense of~\cite[p.~7]{Sch4} and~\cite[Def.~2, p.~319]{H}.  

In \cite[Ex.\ 3.b), p.~67]{Dieu2} the space $\cK(\bR) = \cD^0(\bR) = \cK$ of continuous functions with compact support is equipped with the projective limit topology generated by the seminorms $p_{\varphi}$, $\varphi \in \cE^0$ or $\cE^0_+$, defined by
\[ \cK \ni f \mapsto p_\varphi(f) = \norm{\varphi f}_{\infty}. \]
It coincides with the topology of the strict inductive limit
\[ \varinjlim \cK_K, \qquad K \subset \bR \textrm{ compact}, \]
see \cite[pp.~95--99]{Sch5}. Thus, the strict inductive limit topology of $\cK(\bR)$ is described as the topology of a projective limit (compare Proposition~\ref{prop20}). The space $\cK$ is embedded into $\cL(\cE^0, \cK)$ by the mapping
\begin{align*}
  \cK & \to \cL(\cE^0, \cK), \\
  f &\mapsto [\varphi \mapsto \varphi \cdot f].
\end{align*}
Due to partial continuity and hypocontinuity of the bilinear mapping
\begin{align*}
  \cK \times \cE^0 & \to \cK, \\
  (f, \varphi) & \mapsto f \cdot \varphi 
\end{align*}
the topology of $\cK$ is the induced topology of $\cL_s(\cE^0, \cK)$ or of $\cL_b(\cE^0, \cK)$. Note that $\cK$ is a strict (LB)-space and a (DF)-space.

Projective descriptions of inductive limits are systematically investigated for weighted spaces of \emph{continuous} functions on completely regular Hausdorff spaces in \cite{BMS}.

Projective descriptions of spaces of \emph{differentiable} functions, usually occurring in the theory of distributions, are given in \cite[pp.~98--99]{Sch5}: $\cE^m$, $m \in \bN_0 \cup \{ \infty \}$, $\cS$, $\cO_M$, $\cD^m$, $m \in \bN_0$, $\cB^m$ and $\dot\cB^m$, $m \in \bN_0 \cup \{\infty\}$.

A projective description of the strict (LF)-topology of $\cD(\Omega)$ is given in \cite[p.~21]{GL} and in \cite[Example 7, p.~171]{H}.

Let us treat as an example the topology of the space $\cO_M$ of multipliers from $\cS$ to $\cS$ or from $\cS'$ to $\cS'$. This topology is generated by the seminorms $p_{\varphi, m}$, $\varphi \in \cS$, $m \in \bN_0$, defined by
\[
  \cO_M \ni f \mapsto p_{\varphi, m}(f) = \sup_{\abso{\alpha} \le m} \norm{\varphi\, \pd^\alpha f}_{\infty}
\]
\cite[$3^\circ$, p.~99]{Sch5}. A.~Grothendieck \cite[Chap.~II, p.~130]{G} remarks that this topology is the induced topology of the spaces $\cL_s(\cS, \cS)$ or $\cL_b(\cS, \cS)$---see also~\cite[p.~246]{Sch1}---if $\cO_M$ is embedded into $\cL(\cS, \cS)$ by the multiplication mapping
\[ f \mapsto [\varphi \mapsto \varphi \cdot f ] \]
due to the partial continuity and hypocontinuity of the multiplication
\[ \cO_M \times \cS \xrightarrow{\cdot} \cS. \]
The ``finite order versions'' of the space $\cO_M$ are J.~Horv\'ath's spaces $\cO_C^m = \cS^m_{-\infty}$, $m \in \bN_0$ \cite[Ex.~11, p.~90 and Ex.~9, p.~173]{H}. They are non-strict (LB)-spaces, i.e.,
\[ \cO_C^m = \varinjlim_{k} (1+ \abso{x}^2)^k \cS_0^m = \varinjlim_k (1 + \abso{x}^2)^k \dot\cB^m = \varinjlim_k (\dot \cB^m)_{-k}. \]
A projective description is given in \cite[$m=0$: Prop.~1, $m \in \bN_0$: Prop.~2]{OW} in terms of the seminorms $p_\varphi$, $\varphi \in \cS$, defined by
\[ \cO_C^m \ni f \mapsto p_\varphi(f) = \sup_{\abso{\alpha} \le m} \norm{\varphi\, \pd^\alpha f}_{\infty}. \]
Obviously, the topology of $\cO_C^m$ is the induced topology of the space $\cL_s(\cS, \dot\cB^m)$ or $\cL_b(\cS, \dot\cB^m)$ if $\cO_C^m$ is embedded into $\cL(\cS, \dot\cB^m)$ by the mapping
\[ \cO_C^m \to \cL(\cS, \dot\cB^m), \quad f \mapsto [\varphi \mapsto f \cdot \varphi ]. \]
Furthermore, we mention that we describe in \cite[Prop.~2.5 and Prop.~3.2]{BNO} the topology of the spaces
\[ \cD_{L^p, c},\ 1 \le p \le \infty,\qquad L^p_c,\ 1 < p \le \infty,\qquad \cM^1_c \]
by ``function''-seminorms. The subscript ``small $c$'' denotes the $\kappa$-topology, i.e., the topology of uniform convergence on compact subsets of $\dot\cB$ (if $p=1$) and $\cD'_{L^q}$ ($p>1$, $\frac{1}{p} + \frac{1}{q} = 1$), $L^q$ ($\frac{1}{p} + \frac{1}{q} = 1$) and $\cC_0$, respectively.

If $p = \infty$ this was first done in \cite[(3.5) Cor., p.~71]{DD} by generalising the so-called Buck's or strict topology on $\cB^0$, i.e., $\cB^0_b \coleq (\cB^0, \tau(\cB^0, \cM^1))$. Buck's topology on $\cB^0$ implies that
\[ (\cB^0_b)' = \cM^1, \]
whereas $\dot\cB_c = (\dot\cB, \kappa(\dot\cB, \cD'_{L^1})) = (\dot\cB, \tau(\dot\cB, \cD'_{L^1})) $ yields
\[ (\dot\cB_c)' = \cD'_{L^1} \]
algebraically and topologically, cf. Proposition~1.2.1 in~\cite[p.~6--7]{DVAF}.

Therefore, our task is the description of the topologies $\beta(E,F)$, $\kappa(E,F)$ and $\tau(E,F)$ by ``func\-tion''-seminorms for $E$ a space of functions or distributions and $(E,F)$ is a dual pair. These topologies are defined by seminorms $p_B$ and $p_C$, where $B$ and $C$ are absolutely convex subsets of $F$, $B$ bounded, $C$ compact or weakly compact:
\begin{gather*}
  E \ni e \mapsto p_B(e) = \sup_{f \in B} \abso{\langle e, f \rangle}, \\
  E \ni e \mapsto p_C(e) = \sup_{f \in C} \abso{\langle e, f \rangle}.
\end{gather*}
In contrast to these definitions, ``function''-seminorms do not involve subsets $B,C$.

In order to illustrate this point, recall the above mentioned description of the topology of the multiplicator space $\cO_M$ by the seminorms $p_{\varphi,\alpha}$, $\varphi \in \cS$, $\alpha \in \bN^n_0$:
\[ \cO_M \ni f \mapsto p_{\varphi, \alpha}(f) = \norm{\varphi\, \pd^\alpha f}_{\infty}, \]
equivalent to the topology described by the seminorms $p_B$, $B \subseteq \cO'_M$ bounded:
\[ p_B(f) = \sup_{S \in B} \abso{ {}_{\cO_M} \langle f, S \rangle_{\cO_M'} }, \]
wherein $\cO_M'$ is the space of ``very rapidly decreasing'' distributions \cite[Chap.~II, p.~130]{G}. The equivalence of both descriptions is a simple consequence of the characterisation of bounded subsets $B$ of $\cO_M'$:
\[ B \subseteq \cO_M' \textrm{ bounded } \Longleftrightarrow \exists \varphi \in \cS\ \exists \alpha \in \bN_0^n: B \subseteq \pd^\alpha (\varphi B_{1,1} ), \]
due to the inequality
\begin{align*}
  p_{\varphi,\alpha}(f) & = \norm{\varphi \pd^\alpha f}_{\infty} = \sup_{h \in B_{1,1}}\abso{{_{L^1}\langle} h, \varphi \pd^\alpha f \rangle_{L^\infty} } = \sup_{h \in B_{1,1}} \abso{{}_{\cO_M} \langle f, \pd^\alpha(\varphi h)\rangle_{\cO_M'}} \\
  &\ge \sup_{S \in B} \abso{ \langle f, S \rangle } = p_B(f).
\end{align*}
The opposite inequality is a consequence of the composition of the two continuous mappings
\[ L^1 \to L^1_\infty \coleq \bigcap_{k=0}^\infty (1 + \abso{x^2})^{-k} L^1,\quad, f \mapsto f \cdot \varphi,\quad \varphi \in \cS \]
and
\[ \pd^\alpha \colon L^1_\infty \to \cO_M'. \]
Let us remark that the above characterisation of bounded sets in $\cO_M'$ corresponds to the representation of $\cO_M'$ as a non-strict inductive limit,
\[ \cO_M' = \varinjlim_{\alpha} \pd^\alpha L^1_\infty = \varinjlim_{\alpha, \varphi} \pd^\alpha(\varphi L^1) \]
which is a regular and countable inductive limit of Fr\'echet spaces.

In Section \ref{sec2} we characterise bounded subsets of $\cS$ and we describe the topology of $\cS'$ by "function"-seminorms.

In Section \ref{sec3} we characterise bounded subsets of $\cO_C$ and $\cO_M'$ and we describe the topologies of $\cO_C'$ and $\cO_M$ by ``function''-seminorms.

Similarly, bounded and relatively compact subsets as well as seminorms in the spaces $\cD_{L^p}$, $\cD_{L^p, c}$, $1 \le p < \infty$, $\cD'_{L^q}$, $\cD'_{L^q, c}$, $1 < q \le \infty$, $\dot\cB$, $\dot \cB_c$, $\cD'_{L^1}$, $\cD'_{L^1, c}$, $\dot\cB'$, $\dot \cB'_c$ are described in Section \ref{sec4}.

In Section \ref{sec5} we give a projective description of the topologies of the (LB)-spaces of polynomially increasing functions or measures:
\[ (L^p)_{-\infty},\quad (\cM^1)_{-\infty},\quad \cO_C^m = (\dot \cB^m)_{-\infty} = \cS^m_{-\infty}, \quad m\in\mathbb{N}_0.\]

In Section~\ref{sec6} seminorms in spaces of functions and measures with compact support are presented:
\[ \cD^m(\Omega),\quad L^p_{\textrm{comp}},\quad \cM_{\textrm{comp}},\quad  m\in\mathbb{N}_0 \]

In Section~\ref{sec7} we correct an error in \cite[Proposition 6.4]{BNO}. We characterise \emph{weakly} relatively compact subsets of $\cM^1$ and we prove a generalisation of Buck's topology by a projective description of the spaces $(\cB^m, \tau(\cB^m, \cD'^m_{L^1}))$ and $(\cO_C^m, \tau(\cO_C^m, \cO^{\prime m}_C))$, $m \in \bN_0$. Note that $\cB = \varprojlim_m \cB^m$, $\tau(\cB, \cD'_{L^1}) = \kappa(\cB, \cD'_{L^1}) \subsetneq \beta (\cB, \cD'_{L^1})$, $\cO_M = \varprojlim_m \cO_C^m$ and $\tau(\cO_M, \cO'_M) = \kappa(\cO_M, \cO'_M) = \beta(\cO_M, \cO'_M)$.

\textbf{Notation.} We denote by $\bN_0$ and $\bN$ the sets of natural numbers including $0$ or not. $\alpha \in \bN_0^n$ are multiindices and $\alpha + \ell = (\alpha_1 + 1, \alpha_2 + 1, \dotsc, \alpha_n + 1)$. We use "seminorms" always in the sense ``seminorms generating a topology''. Sometimes we write ``compact'' instead of ``relatively compact''.

Generally, we use for spaces the notation in \cite{Sch5}, \cite{Sch1} and \cite{H}. $\dot\cB'$ is not the dual of $\dot\cB$ but the space of distributions vanishing at infinity defined as the closure of $\cE'$ in $\cD'_{L^\infty}$. We have $\dot\cB^0 = \cC_0 = \cS_0^0$, $\cB^0 = \mathscr{BC}$, wherein $\cS^m_k$ is defined in \cite[p.~90]{H}. The ``argument'' of the spaces is $\bR^n$ and generally it is omitted, i.e.,
\[ \dot\cB^0 = \cC_0 = \cC_0(\bR^n),\quad \cS^m_k = \cS_k^m(\bR^n). \]
Only in Section \ref{sec6} the ``argument'' is an open subset of $\bR^n$. The subscript $\Box_+$ means ``positive functions''. $\abso{x}$ is the Euclidean norm of $x \in \bR^n$, i.e., $\abso{x} = (x_1^2 + \dotsc + x_n^2)^{1/2}$ if $x = (x_1, \dotsc, x_n)$. $\check \varphi$ denotes the function $x \mapsto \varphi(-x)$. The spaces of functions and measures with compact support $L^p_{\textrm{comp}}$, $\cM_{\textrm{comp}}$ are defined by
\[ L^p_{\textrm{comp}} = L^p \cap \cE',\quad \cM_{\textrm{comp}} = \cM \cap \cE'. \]
$\cM^1$ is the space of integrable measures (often called ``bounded measures'') which can be defined by $\cM^1 = (\cC_0)'$. Furthermore, we use the abbreviations $(L^p)_k = (1 + \abso{x}^2)^{-k} L^p$, $k \in \bZ$, $(L^p)_\infty = \bigcap_{k=0}^\infty (L^p)_k, (L^p)_{-\infty} = \bigcup_{k=0}^\infty (L^p)_{-k}$, the $L^p$-norms $\norm{f}_p = (\int_\Omega \abso{f(x)}^p \,\ud x)^{1/p}$ and the unit balls $B_{1,p} = \{ f \in L^p: \norm{f}_p \le 1 \}$, $B_{1,1} = \{ f \in L^1: \norm{f}_1 \le 1\}$ or $B_{1,1} = \{ \mu \in \cM^1: \norm{f}_1 \le 1 \}$.

We use the notation of $B_R$ for the closed Euclidean ball of radius $R$ centred in the origin.

For the space of linear and continuous mappings from a locally convex Hausdorff space $E$ into an other such space $F$ we adopt the notation $\cL(E,F)$. The subscripts ``s'', ``b'' and ``c'' in $\cL_s(E,F)$, $\cL_b(E,F)$ and $\cL_c(E,F)$, respectively, denote the weak, the strong topology and the $\kappa$-topology, in contrast to the designation $\sigma(E,F)$, $\beta(E,F)$ and $\kappa(E,F)$.

As in \cite{BNO} we use the Bessel kernels
\[ L_\lambda = \cF^{-1} ( (1 + \abso{x}^2)^{-\lambda/2}),\quad \lambda \in \bC. \]

\section{Bounded subsets of $\cS$ and seminorms in $\cS'$}\label{sec2}

First, we state a lemma which has its origin in the proofs of~\cite[Lemma 3.6, p.~127]{Chevalley}, \cite[Chap.~II, Lemme 18, pp.~132--133]{G}, and in a statement in~\cite[p.~330]{Roider} and~\cite[(3.2) Lemma]{D}.

\begin{lemma}\label{lemma1}
  Let $F$ be one of the Banach spaces $L^p$, $\cC_0$, $\cM^1$, $1 \le p \le \infty$, and let $F_\infty$ be the Fr\'echet space
  \[ F_\infty = \bigcap_{k=0}^\infty (1+\abso{x}^2)^{-k} F = \varprojlim_k (1+\abso{x}^2)^{-k} F \]
  and $B \subset F_\infty$. 
  Then we have for $F=L^p$:
  \[
    B \text{ bounded in } F_\infty \Longleftrightarrow (\exists \varphi \in \cS_+: B \subset \varphi B_{1,p}).
  \]
  If $F = \cC_0$ or $\cM^1$ the characterisation of bounded subsets of $(\cC_0)_\infty = (\dot \cB^\circ)_\infty = \cS^0_\infty$ and of $(\cM^1)_\infty = (\cD^{\prime 0}_{L^1})_\infty = \cS^{\prime 0}_\infty$ has the same form if we set $B_{1,\infty} = \{ \varphi \in \cC_0\ :\ \norm{\varphi}_{\infty} \le 1 \}$ and $B_{1,1} = \{ \mu \in \cM^1\ :\ \norm{\mu}_1 \le 1\}$, respectively.
\end{lemma}

\begin{proof}
  We only give a proof for $F = L^p$, $1 \le p < \infty$.

  ``$\Leftarrow$'' follows from the continuity of the multiplication
  \[ \cS \times L^p \to L^p_\infty, \quad (\varphi, f) \mapsto \varphi \cdot f. \]
  ``$\Rightarrow$'': The existence of $\varphi \in \cS_+$ has to be shown. We modify B. Roider's proof \cite[p.~331]{Roider} which in turn is a modification of A.~Grothendieck's proof \cite[Chap. II, p.~132--133]{G}.

  Let $Q = [-1/2, 1/2]^n$ be the cube in $\bR^n$ centred at $0$ and $\chi_Q$ its characteristic function, i.e.,
  \[
    \chi_Q(x) = \left\{ \begin{aligned} &1 &\quad&\textrm{if }x \in Q, \\&0&\quad &\textrm{if }x \not\in Q.
      \end{aligned}\right.
  \]
  If $B$ is bounded in $(L^p)_\infty$ we define the sequence $(a(N))_{N \in \bZ^n}$ by putting
  \[ a(N) = ( \sup_{f \in B} ( \abso{f}^p * \chi_Q)(N))^{1/p}. \]
  The boundedness of $B$ implies the existence of constants $c_k$ such that
  \[ \sup_{f \in B} \norm{(1 + \abso{x}^2)^k f}_p \le c_k, \]
  and, more explicitly,
  \[ \sup_{f \in B} \int_{\bR^n} (1+\abso{x}^2)^{kp} \abso{f(x)}^p \,\ud x \le c_k^p. \]
  We show that the sequence $(a(N))_{N \in \bZ^n}$ is rapidly decreasing by estimating the last inequality from below: if $f \in B$, $N = (N_j) \in \bZ^n$, $\tau_N Q = \{ x \in \bR^n: \abso{x_j - N_j} \le 1/2, j = 1 \dotsc, n \}$ we have for $k\in\mathbb{N}_0$ that
  \[
    \int_{\bR^n} (1 + \abso{x}^2)^{kp} \abso{f(x)}^p \,\ud x \ge \int_{\tau_N Q} (1 + \abso{x}^2)^{kp} \abso{f(x)}^p \,\ud x.
  \]
  Moreover, due to $1+|x|^2 \ge \frac{1}{2n}(1+|N|^2)$ for $x \in \tau_N Q$ we obtain
  \[
    \int_{\bR^n} (1 + \abso{x}^2)^{kp} \abso{f(x)}^p \,\ud x \ge \left( \frac{1+\abso{N}^2}{2n}\right)^{kp} \int_{\tau_N Q} \abso{f(x)}^p \,\ud x
  \]
  and
  \[ a(N)(1+\abso{N}^2)^k \le (2n)^k c_k \quad \forall N, \]
  i.e.,
  \[ (a(N))_{N \in \bZ^n} \in s(\bZ^n). \]
  The function
  \[ e^{-\abso{N-x}^2} \colon \bZ^n \times \bR^n \to \bR,\quad
    (N,x) \mapsto e^{-\abso{N-x}^2} \]
  fulfils
  \[ e^{-\abso{N-x}^2} \in s_N' \widehat\otimes \cS_x = s_N'(\cS_x) \]
  due to
  \[ e^{-\abso{x-y}^2} \in \cS_y' \widehat\otimes \cS_x. \]
  Therefore, we obtain
  \[ \sum_{N \in \bZ^n} a(N) e^{-\abso{N-x}^2} \in \cS_x \]
  as a consequence of the well-definedness of the vector-valued scalar product
  \[ \langle\ , \ \rangle \colon s \times s'(\cS) \to \cS \]
  (cf.~\cite[Prop.~4, p.~41]{Sch4a} and \cite[Thm.~7.1, p.~31]{Tata}).
  Defining the function $\varphi \in \cS_+$ by
  \[ \varphi(x) = c \sum_{N \in \bZ^n} a(N) (1+\abso{N}^2)^{n/p} e^{-\abso{N-x}^2}, \]
  $c$ a suitable constant, we prove the desired inclusion
  \[ B \subset \varphi B_{1,p} \]
  by estimating $\norm{f/\varphi}_p$ for $f \in B$:
  \begin{align*}
    \norm{\frac{f}{\varphi}}^p_p &= \int_{\bR^n} \abso{\frac{f}{\varphi}}^p\,\ud x = \sum_{N \in \bZ^n} \int_{\tau_N Q} \abso{\frac{f}{\varphi}}^p \,\ud x \le \sum_{N \in \bZ^n} \frac{1}{\inf_{x \in \tau_N Q} \varphi(x)^p} \underbrace{\int_{\tau_N Q} \abso{f}^p \,\ud x}_{=(\abso{f}^p * \chi_Q)(N) \le a(N)^p}.
  \end{align*}
  Due to $\varphi(x) \ge c a(N)(1+\abso{N}^2)^{n/p} e^{-\abso{N-x}^2}$ and $\varphi(x)^p \ge c^p e^{-p/4} (1+\abso{N}^2)^n a(N)^p$ if $x \in \tau_N Q$ by setting
  \[
    c = e^{1/4} \left( \sum_{N \in \bZ^n} (1+\abso{N}^2)^{-n}\right)^{1/p}
  \]
  we obtain
  \[
    \norm{\frac{f}{\varphi}}^p_p \le \frac{e^{p/4}}{c^p} \sum_{N \in Z^n} (1 + \abso{N}^2)^{-n} \le 1
  \]
  which finishes the proof.
\end{proof}

\begin{remark}
  Another proof of Lemma~\ref{lemma1} can be given by a modification of the proof in~\cite{Wawak}.
\end{remark}

Bounded subsets of $\cS$ are characterised in~\cite[p.~235]{Sch1}. The following characterisation seems to be new.

\begin{proposition}\label{prop1}
  Let $1 \le p \le \infty$.
  \begin{enumerate}[label=(\roman*)]
  \item\label{prop1uno} Let $B \subset \cS$. Then $B$ is bounded in $\cS$ if and only if $\exists \varphi, \psi \in \cS$: $B \subset \varphi * (\psi B_{1,p}))$.
  \item\label{prop1due} The seminorms $p_{\varphi,\psi}$, $\varphi, \psi \in \cS$, defined by
    \[ p_{\varphi, \psi}(S) = \norm{\psi(\varphi * S)}_p,\quad S \in \cS', \]
    generate the topology of $\cS'$.
  \end{enumerate}
\end{proposition}

\begin{proof}
  \ref{prop1uno} The implication ``$\Leftarrow$'' is a consequence of the composition of the two continuous mappings
  \[
    L^p \times \cS \xrightarrow{\cdot} L^p_\infty, \quad L^p_\infty \times \cS \xrightarrow{*} \cS.
  \]

  ``$\Rightarrow$'': \cite[Lemma 1, p.~529]{M} implies the existence of a bounded set $D$ in $\cS$ and a function $\varphi \in \cS$ such that $B = \varphi * D$. Because $D$ is also bounded in $L^p_\infty$, Lemma~\ref{lemma1} yields $\psi \in \cS$ such that $D \subset \psi B_{1,p}$.

  \ref{prop1due} The topology generated by the seminorms $p_{\varphi, \psi}$ is coarser than $\beta(\cS', \cS)$ because the mappings $\cS' \to L^p$, $S \mapsto \psi(\varphi *S)$ are continuous. Conversely, if $B \subset \cS$ is bounded then we have $B \subset \varphi * (\psi B_{1,q})$, $1/q + 1/p = 1$. Thus,
  \begin{align*}
    p_B(S) &= \sup_{\chi \in B} \abso{\langle \chi, S \rangle} \le \sup_{\norm{f}_q \le 1} \abso{\langle \varphi * (\psi f ), S \rangle} = \sup_{\norm{f}_q \le 1} \abso{\langle f, \psi (\check \varphi* S)\rangle} = \norm{\psi  (\check \varphi * S)}_p,
  \end{align*}
  i.e.,
  \[ (\cS', \{ p_{\varphi, \psi} \}) \xrightarrow{\Id} (\cS', \beta(\cS', \cS)) \]
  is continuous.
\end{proof}

\section{Bounded sets in $\cO_C$ and $\cO_M'$. Seminorms in $\cO_C'$ and $\cO_M$.}\label{sec3}

\begin{proposition}\label{prop2}
  \begin{enumerate}[label=(\roman*)]
  \item\label{prop2uno} Let $B \subset \cO_C$. Then
    \[ B \textrm{ bounded in }\cO_C \Longleftrightarrow \exists \alpha \in \bN_0^n\ \exists \varphi \in \cS: B \subset x^\alpha ( \varphi * B_{1,2} ). \]
  \item\label{prop2due} Let $B \subset \cO_M'$. Then
    \[ B \textrm{ bounded in }\cO'_M \Longleftrightarrow \exists \alpha \in \bN_0^n\ \exists \varphi \in \cS: B \subset \pd^\alpha ( \varphi B_{1,2}). \]
  \item\label{prop2tre} The seminorms $p_{\alpha, \varphi}$, $\alpha \in \bN_0^n$, $\varphi \in \cS$, defined by
    \[ \cO_C' \ni S \mapsto p_{\alpha, \varphi}(S) = \norm{\varphi * (x^\alpha S)}_2, \]
    generate the topology of $\cO_C'$.
  \item\label{prop2quattro} The seminorms $q_{\alpha, \varphi}$, $\alpha \in \bN_0^n$, $\varphi \in \cS$, defined by
    \[ \cO_M \ni f \mapsto q_{\alpha, \varphi}(f) = \norm{\varphi \pd^\alpha f}_2, \]
    generate the topology of $\cO_M$.
  \end{enumerate}
\end{proposition}

\begin{proof}
  \ref{prop2tre}, \ref{prop2quattro} and \ref{prop2due} are treated in the introduction, \ref{prop2uno} is a consequence of \ref{prop2due} by applying the Fourier transformation.
\end{proof}

\begin{remark}
  \begin{enumerate}[label=(\arabic*)]
  \item The seminorms in (iii) are presented also in \cite[4.2 Lemma]{SLT}.
  \item By simple manipulations all $L^2$-norms can be replaced by $L^p$-norms, $1 \le p \le \infty$.
  \item In \cite[Lemma 2, p.~530]{M}, bounded sets $B$ in $\cO_C$ are characterised by the existence of a bounded set $D \subseteq \cS'$ and a function $\varphi \in \cS$ such that $B = \varphi * D$. This characterisation immediately yields that the seminorms
    \[
      S \mapsto \norm{(1 + \abso{x}^2)^k (\varphi * S)}_\infty, \qquad S \in \cO_C', \;\varphi \in \cS, \; k \in \bN_0
    \]
    generate the topology of $\cO_{C}'$.
  \item The problem of finding seminorms in terms of classical norms and convolutions and multiplications with functions which generate the topology of $\cO_C$ remains open (cf.~\cite[p.~442]{H}).
  \end{enumerate}
\end{remark}

\section{Bounded and compact sets, seminorms in the spaces $\cD_{L^{\lowercase{p}}}$, $\cD_{L^{\lowercase{p}}, c}$, $\dot\cB$, $\dot\cB_c$, $\cD'_{L^{\lowercase{q}}}$, $\cD'_{L^{\lowercase{q}}, c}$, $\dot\cB'$, $\dot\cB'_c$}\label{sec4}

As in \cite[p.~5]{Sch4} the index $c$ in $\cD_{L^p,c}$, $\dot\cB_c$, $\cD'_{L^q, c}$, $\dot\cB'_c$ refers to the topology of uniform convergence on absolutely convex, compact subsets of $\cD'_{L^q}$, $\cD'_{L^1}$, $\cD_{L^p}$, $\cD_{L^1}$, respectively, $1<p<\infty$, $\frac{1}{p}+\frac{1}{q}=1$.

\begin{proposition}\label{prop3}
  Let $1 \le p < \infty$ and $\frac{1}{p} + \frac{1}{q} = 1$.
  \begin{enumerate}[label=(\roman*)]
  \item\label{prop3uno} The bounded subsets of $\cD_{L^p}$ and $\cD_{L^p, c} = (\cD_{L^p}, \kappa(\cD_{L^p}, \cD'_{L^q}))$ coincide with the relatively compact subsets of $\cD_{L^p, c}$.
  \item\label{prop3due} They are characterised by:
    \[
      B \subset \cD_{L^p} \textrm{ bounded } \Longleftrightarrow \exists m \in \bN\ \exists \varphi_1,\ldots,\varphi_m \in \cD: B \subset \sum_{i=1}^{m} \varphi_i * B_{1,p}
    \]
  \item\label{prop3tre} The seminorms $p_\varphi$, $\varphi \in \cD$, defined by
    \[
      \cD'_{L^q} \ni S \mapsto p_\varphi(S) = \norm{\varphi * S}_q,
    \]
    generate the topology of $\cD'_{L^q}$, i.e., $\beta(\cD'_{L^q}, \cD_{L^p})$.
  \item \label{prop3quatro} The seminorms $p_\varphi$, $\varphi \in \cS$, defined by
    \[
      \cD'_{L^q} \ni S \mapsto p_\varphi(S) = \norm{\varphi * S}_q,
    \]
    generate the topology of $\cD'_{L^q}$, i.e., $\beta(\cD'_{L^q}, \cD_{L^p})$.
  \end{enumerate}
\end{proposition}

\begin{proof}
  \ref{prop3uno} follows by \cite[Prop. 5.1, p.~1701 and Prop. 5.3, p.~1702]{BNO}.

  \ref{prop3due} ``$\Leftarrow$'' is a consequence of the continuity of the convolution mapping $\cD \times L^p \xrightarrow{*} \cD_{L^p}$.

  ``$\Rightarrow$'': We first show the one-dimensional case. Let $B\subset \cD_{L^p}$ be bounded. We set
  \[
    c_m := \sup_{f\in B} \|f^{(2m)}\|_{p}
  \]
  and use Lemme~2.5 in~\cite[p.~309]{DM1978} to obtain a sequence $(\alpha_m)_{m=0}^{\infty}$ of positive numbers and two test functions $g,h\in \cD$ such that
  \[
    \alpha_m \leq \frac{1}{2^{m+1}c_m}\qquad\text{and}\qquad \sum_{m=0}^{N}(-1)^m \alpha_m \delta^{(2m)}*g \to \delta +h 
  \]
  in $\cE'$ for $N\to\infty$. Since the convolution $\cE'\times \cD_{L^p}\to \cD_{L^p}$ is hypocontinuous, we have that
  \[
    \sum_{m=0}^{N}(-1)^m \alpha_m \delta^{(2m)}*g*\psi \to \psi +h*\psi
  \]
  in $\cD_{L^p}$ for $N\to\infty$ uniformly with respect to $\psi \in B$. Moreover for $\psi\in B$ note that
  \begin{align*}
    \Big\|\sum_{m=0}^{N}(-1)^m \alpha_m \delta^{(2m)}*\psi&-\sum_{m=0}^{M} (-1)^m\alpha_m\delta^{(2m)}*\psi \Big\|_{p} \leq \sum_{m=N+1}^{M} \alpha_m \|\psi^{(2m)}\|_{p} \leq \sum_{m=N+1}^{M} 2^{-m-1}
  \end{align*}
  for $M\ge N$ which shows that $\sum_{m=0}^{N}(-1)^m \alpha_m \delta^{(2m)}*\psi\to \chi\in B_{1,p}$ for $N\to \infty$. In other words, we obtained a representation 
  \[
    \psi = g*\chi -  h*\psi  \in g*B_{1,p} + C h *B_{1,p}
  \]
  for all $\psi\in B$ and some $C>0$ depending only on $h$ and $c_0$ as required. For $n>1$, we use that $\delta(x) = \delta(x_1) \otimes\cdots\otimes \delta(x_n)$ and set
  \[
    c_m := \max\Big\{1, \sup_{|\alpha|\leq 2m, f\in B} \|\partial^\alpha f\|_{p}\Big\}
  \]
  which allows us to conclude the $n$-dimensional case from the one-dimensional case.

  \ref{prop3tre} The inequality
  \begin{align*}
    p_B(S) &= \sup_{\psi \in B} \abso{ \langle \psi, S \rangle } \le \sup_{f_1,\ldots,f_m \in B_{1,p}} \abso{ \langle \sum_{i=1}^{m} \varphi_i * f_i, S \rangle } \le \sum_{i=1}^{m} \sup_{f \in B_{1,p}} \abso{ \langle \varphi_i * f, S \rangle } = \sum_{i=1}^{m} \norm{\check \varphi_i * S}_q
  \end{align*}
  shows that the topology generated by the seminorms $p_\varphi$, $\varphi \in \cD$, is finer than the topology $\beta(\cD'_{L^q}, \cD_{L^p})$. The converse inequality follows from the continuity of
  \[ \cD'_{L^q} \to L^q, \quad S \mapsto \varphi * S, \]
  which is the transpose of the continuous mapping $L^p \to \cD_{L^p}, \; \psi \mapsto \check \varphi * \psi$.

  \ref{prop3quatro} Since $\cD \subset \cS$ this assertion is a direct consequence of \ref{prop3tre} together with the hypocontinuity of the convolution mapping
  \[
    \cS\times \cD'_{L^q} \to L^q, \qquad (\psi,S)\mapsto \psi*S
  \]
  which follows from $\cS\subset \cD_{L^1}$, see also~\cite{Larcher}.
\end{proof}

\begin{remarks}
  \begin{enumerate}
  \item Recall that the spaces $\cD'_{L^q}$ are isomorphic with the convolutor spaces
    \[
      \cO_{C}'(\cD,L^q) = \cO_{C}'(\cS, L^q) = \cO_{C}'(\cD,\cD'_{L^q})=\cO_{C}'(\cS,\cD'_{L^q})
    \]
    as defined in~\cite[p.~72]{Sch4} as vector spaces, cf.~\cite[p.~201]{Sch1}. Proposition~\ref{prop3}, and Proposition~\ref{prop8} for~$q=1$, show that if we equip these spaces with the topologies induced by $\cL_s(\cD,L^q)$, $\cL_s(\cS,L^q)$, $\cL_s(\cD,\cD_{L^q}')$, $\cL_s(\cS,\cD_{L^q}')$, which on the convolutor spaces agree with the ones induced by $\cL_b(\cD,L^q)$, $\cL_b(\cS,L^q)$, $\cL_b(\cD,\cD_{L^q}')$, $\cL_b(\cS,\cD_{L^q}')$, the above mentioned isomorphisms become isomorphisms in the sense of locally convex spaces. In particular the convolutor spaces are ultrabornological since $\cD_{L^q}'$ is ultrabornological. In the case of $\cO_{C}'(\cD,L^1)$ this is also a special case of Theorem~1.1 in~\cite[p.~830]{DV21}.
  \item Proposition~\ref{prop3}, \ref{prop3due}--\ref{prop3quatro} could also be formulated for spaces $\cD_{E}$ for translation invariant Banach spaces $E$ in the sense of~\cite{DPV2015} and the above proof also works in this more general setting. Note that these spaces are intimately related to convolutor spaces. In~\cite{DPV2015} it is shown that for a translation invariant Banach space, the space $\cD_{E'_*}'$ algebraically agrees with the space of convolutors $\cO'_C(\cD,E')$ and has the same bounded sets. Theorem~5.6 in~\cite[p.~134]{Andreas2019} shows the space $\cD_E$ is distinguished and hence $\cD'_{E'_*}$ is barrelled. 
  \item An independent proof of Proposition~\ref{prop3}, \ref{prop3due} for $p=2$ can be given using the Fourier transform and Lemma~\ref{lemma1}: $B\subset \cD_{L^2}$ is bounded if and only there is a $\varphi\in\cD$ or $\varphi\in\cS$ with $B\subset \varphi*B_{1,2}$.
  \item A similar result to Proposition~\ref{prop3}, \ref{prop3due} in the setting of ultradistributions has been given in Section~7 of~\cite{DPV21}.
  \end{enumerate}
\end{remarks}

Next, we characterise relatively compact sets in $\cD_{L^p}$ and describe the topologies of $\cD'_{L^q, c}$ by seminorms:

\begin{proposition}\label{prop4}
  Let $1 \le p < \infty$ and $\frac{1}{p} + \frac{1}{q} = 1$.
  \begin{enumerate}[label=(\roman*)]
  \item\label{prop4uno} Let $C \subset \cD_{L^p}$. Then
    \[
      C \textrm{ relatively compact } \Longleftrightarrow \exists g \in \dot\cB, \exists \varphi_1,\ldots,\varphi_m \in \cD\colon C \subset g\Big(\sum_{i=1}^{m}\varphi_i * B_{1,p}\Big)
    \]
    or, equivalently, $\exists g \in \dot\cB, \exists \varphi_1,\ldots,\varphi_m  \in \cD\colon C \subset \sum_{i=1}^{m} \varphi_i * (g B_{1,p})$.
  \item\label{prop4due} Let $C \subset \cD_{L^p}$. Then
    \[ C \textrm{ relatively compact } \Longleftrightarrow \exists g \in \dot\cB\ \exists D \subset \cD_{L^p}\textrm{ bounded}: C \subset gD. \]
    This characterisation is a factorisation of compact sets in the Fr\'echet space $\cD_{L^p}$ on which the Fr\'echet algebra $\dot\cB$ operates by multiplication.

  \item\label{prop4tre} The seminorms $p_{g, \varphi}$, $(g, \varphi) \in \dot\cB \times \cD$ defined by
    \[
      \cD'_{L^q} \ni S \mapsto p_{g,\varphi}(S) = \norm{g(\varphi * S)}_q
    \]
    generate the topology of $\cD'_{L^q, c}$.
    
    An equivalent description is given by the seminorms
    \[\cD'_{L^q} \ni S \mapsto \norm{\varphi * (gS)}_q.\]
    In both cases, using $\varphi\in \cS$ instead of $\varphi\in \cD$ results in an equivalent system of seminorms.
  \end{enumerate}
\end{proposition}

\begin{proof}
  \ref{prop4due} ``$\Leftarrow$'': By Leibniz' formula we obtain, for all $\alpha \in \bN_0^n$,
  \[ \pd^\alpha C \subset \sum_{\beta \le \alpha} \binom{\alpha}{\beta} \pd^{\alpha-\beta} g \pd^\beta D, \]
  i.e., the boundedness and the uniform smallness of all $\pd^\alpha C$. The $L^p$-equicontinuity of $\pd^\alpha C$ is implied by
  \[ \norm{ \tau_h ( \pd^\alpha C) - (\pd^\alpha C)}_p \le \norm{\pd^{\alpha + \ell} C}_p \cdot \abso{h}, \]
  $h \in \bR^n$, $\ell = (1,\dotsc, 1)$, and the boundedness of $\pd^{\alpha + \ell} C$ in $L^p$. Therefore, $\pd^\alpha C$ is relatively compact by the M.~Fr\'echet-A.~Kolmogorov-M.~Riesz-H.~Weyl-Theorem.

  ``$\Rightarrow$'':
  We want to apply a factorisation theorem for Fr\'{e}chet algebras. First note that by Proposition~1.6 in~\cite[p.~336]{Voigt} the sequence $(\mathrm{e}^{-|x|^2/k^2})_{k=1}^{\infty}$ is a uniformly bounded approximate unit for the Fr\'{e}chet algebra $\dot{\cB}$ with multiplication. If we show that $\mathrm{e}^{-|x|^2/k^2}f\to f$ in $\cD_{L^p}$ for $f\in\cD_{L^p}$ and $k\to\infty$ the claim follows from Corollary~4 in~\cite[p.~149]{VoigtFactorization}. Note that  $\mathrm{e}^{-|x|^2/k^2} \to 1$ in $\cE$ which together with $\cD_{L^p}\subset\dot{\cB}$ implies the needed approximation relation.

  \ref{prop4uno} ``$\Rightarrow$'': Since $C\subset \cD_{L^p}$ is compact, we may use Lemme~2.5 in~\cite[p.~309]{DM1978} to obtain test functions $\varphi_1,\ldots, \varphi_m \in \cD$ and compact sets $C_1,\ldots, C_m \subset L^p$ with
  \[
    C \subset \sum_{j=1}^{m} (\varphi_j* C_j) \subset \sum_{j=1}^{m} \varphi_j * D
  \]
  where $D:=\bigcup_{j=1}^{m} C_j$ which is a compact subset of $L^p$ as well. Again, we use that the sequence $\mathrm{e}^{-|x|^2/k^2}$ is a uniformly bounded approximate unit for the Fr\'{e}chet algebra $\dot\cB$. Arguing as in Lemma~2.6 in~\cite[p.~1698]{BNO}, considering $L^p$ as a Banach module over the Fr\'{e}chet algebra $\dot\cB$ and using Theorem~3 in~\cite[p.~148]{VoigtFactorization}, we may conclude the existence of a function $g\in\dot\cB$ with $D\subset g B_{1,p}$. Summing up, this results in $C \subset \sum_{j=1}^{m} \varphi_j * (g B_{1,p})$, as required.

  The converse implication follows from the continuity of the convolution mapping $\cD\times \cD_{L^p}\to\cD_{L^p}$ together with~\ref{prop4due}.

  \ref{prop4tre} If $C \subset \cD_{L^p}$ is compact and $S \in \cD'_{L^q}$ then
  \begin{align*}
    p_C(S) &\le \sup_{\norm{f}_p \le 1} \abso{\langle g(\varphi * f), S \rangle } = \sum_{i=1}^{m} \sup_{f \in B_{1,p}}\abso{\langle f, \check\varphi_i * (gS) \rangle } = \sum_{i=1}^{m} \norm{\check \varphi_i *(gS)}_q.
  \end{align*}
  Hence, the topology on $\cD'_{L^q}$ generated by the seminorms $p_{g, \varphi}$ is finer than $\beta(\cD'_{L^q}, \cD_{L^p})$. It also is coarser because the mapping
  \[ \cD'_{L^q} \to L^q, S \mapsto \varphi * (gS), \]
  is continuous.

  The equivalent description in \ref{prop4tre} immediately results from the second characterisation of compact sets in \ref{prop4uno}. The equivalence of the systems with $\varphi\in \cS$ follows from Proposition~\ref{prop4}, \ref{prop4tre}.
\end{proof}

In the next proposition we recall topology-generating seminorms on $\cD_{L^p}$ and $\cD_{L^p, c}$:

\begin{proposition}\label{prop5}
  Let $1 \le p \le \infty$ and $L_{-2m} = (1 - \Delta_n)^m \delta$, $m \in \bN_0$.
  \begin{enumerate}[label=(\roman*)]
  \item\label{prop5uno} The topology of $\cD_{L^p}$ is generated by the seminorms $p_m$, $m \in \bN_0$, defined by
    \[ \cD_{L^p} \ni \varphi \mapsto p_m(\varphi) = \sup_{|\alpha|\leq m} \norm{\pd^\alpha \varphi}_p, \]
    or equivalently by
    \[ \cD_{L^p} \ni \varphi \mapsto \norm{L_{-2m} * \varphi}_p. \]
  \item\label{prop5due} The topology of $\cD_{L^p, c}$ is generated by the seminorms $p_{g,m}$, $g \in \cC_0$, $m \in \bN_0$, defined by
    \[ \cD_{L^p} \ni \varphi \mapsto p_{g,m}(\varphi) = \norm{g(L_{-2m} * \varphi)}_p. \]
  \end{enumerate}
\end{proposition}

\begin{proof}
  \ref{prop5uno} follows from the definition of the topology of $\cD_{L^p}$ and from the properties of the Bessel kernels $L_\lambda$.

  \ref{prop5due} see \cite[Propositions 2.5 and 3.2]{BNO}.
\end{proof}

\begin{proposition}\label{prop6}
  Let $1 < q \le \infty$.

  \begin{enumerate}[label=(\roman*)]
  \item\label{prop6uno} The bounded sets in $\cD'_{L^q}$ and $\cD'_{L^q, c}$ coincide with the relatively compact sets in $\cD'_{L^q, c}$.
  \item\label{prop6due} They are characterised by:
    \[ B \subset \cD'_{L^q} \textrm{ bounded } \Longleftrightarrow \exists m \in \bN_0\ \exists C_m>0: B \subset C_m (L_{-2m} * B_{1,q}) \]
  \end{enumerate}
\end{proposition}
\begin{proof}
  \ref{prop6uno} We imitate the reasoning from \cite[pp.~126--127]{Sch4}: bounded sets in $\cD'_{L^q}$ also are bounded in $\cD'_{L^q, c}$. Conversely, bounded sets in $\cD'_{L^q, c}$ also are $\sigma(\cD'_{L^q}, \cD_{L^p})$-bounded ($1/p + 1/q = 1$). Hence, they are $\beta(\cD'_{L^q}, \cD_{L^p})$-bounded because $\cD_{L^q}'$ is complete. They are equicontinuous on $\cD_{L^p}$ and thus also relatively compact in $\cD'_{L^q, c}$

  \ref{prop6due} ``$\Rightarrow$'': the (LB)-space $\cD'_{L^q}$ is complete and hence regular. Therefore, there exist for a bounded set $B$ in $\cD'_{L^q}$ some $m \in \bN_0$ and $C_m>0$ such that
  \[ B \subset C_m ( L_{-2m} * B_{1,q}). \]

  ``$\Leftarrow$'' follows from the continuity of the mappings
  \[ \varphi \mapsto L_{-2m} * \varphi,\quad L^q \to \cD'_{L^q}. \qedhere \]
\end{proof}

\begin{remark}
  The above characterisation is essentially the one given in Remarque 2 in~\cite[p.~202]{Sch1}.
  Other characterisations of the bounded subsets of $\cD'_{L^q}$ are given in Theorem~1 in~\cite[p.~51]{AP1994} and Theorem~3.1 in~\cite[p.~178]{APV2014}.
\end{remark}

In order to characterise compact subsets in $\cD'_{L^q}$ in Proposition~\ref{prop7} below we need

\begin{lemma}\label{lemma2}
  Let $C$ be a compact subset of $L^q$, $1 \le q < \infty$. Then there exists a function $\varphi \in \dot{\cB}$ such that $C \subset \varphi B_{1,q}$ (factorisation of compact sets in $L^q$ over the Fr\'echet algebra $\dot{\cB}$ operating by multiplication).
\end{lemma}

\begin{proof}
  Since by Proposition~1.6 \cite[p.~336]{Voigt} the sequence $(\mathrm{e}^{-|x|^2/k^2})_{k=1}^{\infty}$ is a uniformly bounded approximate unit for the Fr\'{e}chet algebra $\dot{\cB}$ with multiplication we only have to show that $\mathrm{e}^{-|x|^2/k^2}f\to f$ in $L^q$ for $f\in L^q$ and $k\to\infty$ in order to conclude the claim from Corollary~4 in~\cite[p.~149]{VoigtFactorization}. The required convergence follows from the observation that $\mathrm{e}^{-|x|^2/k^2}\to 1$ uniformly on compact sets for $k\to\infty$.
  
\end{proof}

\begin{proposition}\label{prop7}
  Let $1 < q < \infty$ and $C \subset \cD'_{L^q}$. Then the following assertions are equivalent:
  \begin{enumerate}[label=(\roman*)]
  \item\label{prop7uno} $C$ is relatively compact;
  \item\label{prop7due} $\exists (m,\varphi_1,\ldots,\varphi_M) \in \bN_0 \times \dot\cB^{M}$: $C \subset \sum_{j=1}^{M} \varphi_j (L_{-2m} * B_{1,q})$;
  \item\label{prop7tre} $\exists (m,\varphi) \in \bN_0 \times \dot\cB$: $C \subset L_{-2m} * (\varphi B_{1,q})$;
  \item\label{prop7quattro} $\exists \varphi \in \dot\cB$, $\exists D \subseteq \cD'_{L^q}$ bounded: $C \subset \varphi D$.
  \end{enumerate}
\end{proposition}
\begin{proof}
  \ref{prop7uno} $\Rightarrow$ \ref{prop7tre}: By \cite[Proposition 2.4]{BNO}, if $C \subset \cD'_{L^q}$ is relatively compact then there exists $m \in \bN_0$ such that $L_{2m} * C$ is relatively compact in $L^q$. Hence, by Lemma~\ref{lemma2} there exists $\varphi \in \dot\cB$ such that
  \[ L_{2m} * C \subset \varphi B_{1,q} \textrm{ or }C \subset L_{-2m} * (\varphi B_{1,q}). \]

  \ref{prop7tre} $\Rightarrow$ \ref{prop7due}: Leibniz' rule together with Theorem~3 in~\cite[p.~135]{St} or Corollary~1 in~\cite[p.~347]{Mazja} implies:
  \begin{align*}
    C & \subset L_{-2m} * (\varphi B_{1,q}) = (1-\Delta_n)^m (\varphi B_{1,q})  = \sum_{\abso{\alpha}, \abso{\beta} \le 2m} c_{\alpha \beta} \pd^\alpha \varphi \pd^\beta B_{1,q} \\
    &\subset \sum_{\abso{\alpha}\le 2m} \partial^\alpha\varphi \,(1-\Delta_n)^m B_{1,q}.
  \end{align*}  

  \ref{prop7due} $\Rightarrow$ \ref{prop7quattro}: $B_{1,q}$ is bounded in $L^q$ which implies that $L_{-2m} * B_{1,q}$ is a bounded subset of $\cD'_{L^q}$. Since $\{\varphi_1,\ldots,\varphi_M\}$ is a compact subset of $\dot\cB$ we may use the compact factorisation property of $\dot\cB$, see Proposition~1.6 in~\cite[p.~336]{Voigt}, to obtain $\psi\in\dot\cB$ and $\psi_1,\ldots,\psi_M\in\dot\cB$ with $\varphi_j = \psi \psi_j$. Hence
  \[
    C \subset \sum_{j=1}^{M} \varphi_j L_{-2m}*B_{1,q} = \psi \Big(\sum_{j=1}^{M} \psi_j L_{-2m}*B_{1,q}\Big) =: \psi D
  \]
  where $D\subset\cD'_{L^q}$ is a bounded set.

  \ref{prop7quattro} $\Rightarrow$ \ref{prop7uno}: Let $C \subset \varphi D$, $\varphi \in \dot\cB$, $D$ bounded in $\cD'_{L^q}$. Due to the regularity of the (LB)-space $\cD'_{L^q}$ there exist $m \in \bN_0$ and $c>0$ such that $L_{m}*(\varphi D)$ is a bounded subset of $L^q$. We now show that $L_{m+1}*(\varphi D)$ is a compact subset of $L^q$ which implies that $\varphi D$ is a compact subset of $\cD_{L^q}'$ and hence its subset $C$ is relatively compact. First observe that $L_{m+1}*(\varphi D)$ is a bounded subset of $L^q$ and that $\tau_hL_1\to L_1$ in $L^1$ for $h\to 0$ which implies that
  \[
    L_{m+1}*(\varphi D)=L_1*L_{m}*(\varphi D)
  \]
  is $L^q$-equicontinuous. We want to apply Theorem~3 in~\cite[p.~148]{VoigtFactorization} to $L^p$ as module over the Fr\'{e}chet algebra $\dot\cB$ to obtain a function $g\in \dot\cB$ with $L_{m+1}*(\varphi D) \subset g B_{1,p}$ which implies that this set is equitight which completes the requirement to conclude the compactness using the Fr\'echet–Kolmogorov-Riesz-Weyl theorem. In order to apply the aforementioned factorisation theorem, we need to show that $\mathrm{e}^{-\varepsilon |x|^2}(L_{m+1}*(\varphi S))\to L_{m+1}*(\varphi S)$ for $\varepsilon\searrow 0$ uniformly for $S\in D$. First note that if $L_{m+1}*(\varphi S)\in L^q$ the same is true for the distribution $L_{m+1}*(\psi\varphi S)$ whenever $\psi\in \cD_{L^\infty}$, see e.g.~\cite[p.~636]{Belyaev}. We pick a $\psi\in\cD$ which is one on the Euclidean unit ball $B_1$ and whose support is contained in the ball $B_2$ of radius two. We define $\psi_R(x):= \psi(x/R)$ and observe that $\psi_R\varphi \to\varphi$ for $R\to\infty$ in $\dot\cB$ and hence $\varphi\psi_R S \to \varphi S$ in $\cD_{L^p}'$ uniformly for $S\in D$. Using that $L_{m+1}*(\varphi\psi_RS)\in\cO_C'\cap L^q$ a straightforward computation shows the required uniform convergence of $\mathrm{e}^{-\varepsilon |x|^2}(L_{m+1}*(\varphi S))\to L_{m+1}*(\varphi S)$ for $\varepsilon\searrow 0$.
\end{proof}

\begin{remarks}
  \begin{enumerate}
  \item Note that --- in contrast to the three conditions characterising relatively compact sets in $L^q$ in the classical theorem of M.~Fr\'echet-A.~Kolmogorov-M.~Riesz-H.~Weyl --- only one condition for compact sets in $\cD'_{L^q}$, i.e., \ref{prop7quattro} in Proposition~\ref{prop7} (which guarantees the uniform smallness of the elements of a compact set) is necessary to characterise relatively compact sets in $\cD'_{L^q}$.
  \item A characterisation of relatively compact subsets $B$ of the spaces $\cL_s^p := L_s*L^p$, $s\in\bC$, $1\leq p < \infty$) of Bessel potentials is the following:
    \[
      \exists (g,h)\in \cC_0\times L^1\colon B\subset L_s*(h*(gB_{1,p})).
    \]
    For $s\in\bN$ and $1<p<\infty$ this is a new characterisation of relatively compact subsets in the Sobolev space $W^{s,p}$, cf. Cor.~9 in~\cite[p.~390]{HOH} and 26.2.2.~Theorem in~\cite[p.~218]{Besov}.
  \end{enumerate}
\end{remarks}

Next, we investigate bounded and relatively compact subsets and seminorms in the spaces $\dot\cB$, $\dot\cB_c$, $\cD'_{L^1}$, $\cD'_{L^1, c}$, $\dot\cB'$, $\dot\cB'_c$, i.e., the analogues of Propositions \ref{prop3} to \ref{prop7} for these spaces. They occupy a special position because in each case the dual pair defining the topology has to be specified.

Let us recall that $\dot\cB'$ is by definition the closure of $\cE'$ in $\cD'_{L^\infty}$. It can be shown that its dual is the space $\cD_{L^1}$. Hence, we put:
\begin{align*}
  \dot\cB_c &= (\dot \cB, \kappa(\dot \cB, \cD'_{L^1})) = (\dot \cB, \tau(\dot \cB, \cD'_{L^1})), \\
  \cD'_{L^1, c} &= ( \cD'_{L^1}, \kappa(\cD'_{L^1}, \dot\cB)), \\
  \dot\cB'_c &= (\dot\cB', \kappa(\dot\cB', \cD_{L^1})) = (\dot\cB', \tau(\dot\cB', \cD_{L^1})),
\end{align*}
whereby we used the Schur property of the spaces $\cD_{L^1}$ and $\cD'_{L^1}$, cf.~\cite[p.~153--154]{no_schur}.

\begin{proposition}\label{prop8}
  \begin{enumerate}[label=(\roman*)]
  \item\label{prop8uno} The bounded subsets of $\dot\cB$ and $\dot\cB_c$ coincide with the relatively compact subsets of $\dot\cB_c$.
  \item\label{prop8due} They are characterised by: $B \subset \dot\cB$ bounded $\Leftrightarrow$ $\exists \varphi_1,\ldots,\varphi_m \in \cD$: $B \subset \sum_{i=1}^{m} \varphi_i * B_{1,\infty}$ with $B_{1,\infty} = \{ f \in \cC_0: \norm{f}_\infty \le 1 \}$.
  \item\label{prop8tre} The seminorms $p_\varphi$, $\varphi \in \cD$, or $\varphi\in\cS$, defined by \[\cD'_{L^1} \ni S \mapsto p_\varphi(S) = \norm{\varphi * S}_1\] generate the topology of $\cD'_{L^1}$, i.e., $\beta(\cD'_{L^1}, \dot\cB)$.
  \end{enumerate}
\end{proposition}

\begin{proof}
  \ref{prop8uno} Analogously to~\cite[p.~127]{Sch4}, we argue that bounded subsets of $\dot\cB$ also are bounded in $\dot\cB_c$. Conversely, bounded subsets of $\dot\cB_c$ are also $\sigma(\dot\cB, \cD'_{L^1})$-bounded. Due to the completeness of $\dot\cB$, they are also $\beta(\dot\cB, \cD'_{L^1})$-bounded. They are equicontinuous on $\cD'_{L^1}$ and thus also relatively compact in $\dot\cB_c$.

  \ref{prop8due} ``$\Leftarrow$'' is a consequence of the continuity of the convolution mapping $\cD_{L^1} \times \cC_0 \xrightarrow{*} \dot\cB$.

  ``$\Rightarrow$'': The implication follows analogously to the proof of Proposition~\ref{prop3}, \ref{prop3due} using Lemme~2.5 in~\cite[p.~309]{DM1978} and the hypocontinuity of the convolution mapping $\dot{\cB}\times\cE'\to \dot\cB$, see e.g.~\cite{Larcher}.

  \ref{prop8tre} The inequality
  \begin{align*}
    p_B(S) &= \sup_{\psi \in B} \abso{ \langle \psi, S \rangle } \le \sup_{f_1,\ldots,f_m \in B_{1,\infty}} \abso{ \langle \sum_{i=1}^{m} \varphi_i * f_i, S \rangle } \le \sum_{i=1}^{m} \sup_{f \in B_{1,\infty}} \abso{ \langle \varphi_i * f, S \rangle } = \sum_{i=1}^{m} \norm{\check \varphi * S}_1
  \end{align*}
  for $S \in \cD'_{L^1}$ shows that the topology generated by the seminorms $p_\varphi$, $\varphi \in \cD$, is finer than the topology $\beta(\cD'_{L^1}, \dot\cB)$.

  The converse inequality follows from the continuity of the mapping
  \[ \cD'_{L^1} \to L^1, \quad S \mapsto \varphi * S, \qquad\varphi \in \cD. \qedhere \]
\end{proof}

\begin{remark}
  A second proof of Proposition~\ref{prop8}, \ref{prop8due}, \ref{prop8tre} can be given using that the space $\cD'_{L^1}$ is isomorphic to the space of convolutors $\cO_C'(\cD,L^1)$---see~\cite[p.~831]{DV21}---and that on this space of convolutors the spaces $\cL_b(\cD,L^1)$ and $\cL_s(\cD, L^1)$ induce the same topology. 
\end{remark}

The next proposition states the validity of Proposition~\ref{prop6} in the case of $q=1$.

\begin{proposition}\label{prop9}
  \begin{enumerate}[label=(\roman*)]
  \item\label{prop9uno} The bounded subsets of $\cD'_{L^1}$ and $\cD'_{L^1,c}$ coincide with the relatively compact subsets of $\cD'_{L^1,c}$.
  \item\label{prop9due} They are characterised by:
  \[ B \subset \cD'_{L^1} \textrm{ bounded } \Longleftrightarrow \exists m \in \bN_0\ \exists C_m>0: B \subset C_m L_{-2m} * B_{1,1}. \]
  \end{enumerate}
  The bounded sets in $\cD'_{L^1}$ yield the seminorms on $\cD_{L^\infty}$ and on $\dot\cB$ which are known by definition and which are recalled in Proposition~\ref{prop5} \ref{prop5uno}.
\end{proposition}

\begin{proof}
  \ref{prop9uno} Bounded subsets of $\cD'_{L^1}$ are also bounded in $\cD'_{L^1,c}$. Conversely, the bounded subsets of $\cD_{L^1, c}'$ are also $\sigma(\cD'_{L^1}, \dot\cB)$-bounded. Due to the completeness of $\cD_{L^1}'$ they are also $\beta(\cD'_{L^1}, \dot\cB)$-bounded. They are equicontinuous on $\dot\cB$ and, thus, relatively compact in $\cD'_{L^1,c}$.

  \ref{prop9due} ``$\Rightarrow$'': by \cite[Proposition 2.2]{BNO} the space $\cD'_{L^1}$ is compactly regular. Hence, if $B \subset \cD'_{L^1}$ is bounded there exists an $m \in \bN$ such that $B$ is bounded in $L_{-2m} * L^1$. This implies the existence of $C_m>0$ such that
  \[ B \subset C_m ( L_{-2m} * B_{1,1}). \]
  ``$\Leftarrow$'' is a consequence of the continuity of the mapping $L^1 \to \cD'_{L^1}$, $f \mapsto L_{-2m} * f$.
\end{proof}

We characterise next the bounded subsets in $\dot\cB'$ (analogously to Proposition~\ref{prop6}):

\begin{proposition}\label{prop10}
  \begin{enumerate}[label=(\roman*)]
  \item\label{prop10uno} The bounded subsets of $\dot\cB'$ and $\dot\cB'_c$ coincide with the relatively compact subsets of $\dot\cB'_c$.
  \item\label{prop10due} They are characterised by: $C \subset \dot\cB'$ bounded $\Leftrightarrow$ $\exists m \in \bN_0$ $\exists c_m>0$: $C \subset c_m(L_{-2m} * B_{1,\infty})$, with $B_{1,\infty} = \{ f \in \cC_0: \norm{f}_\infty \le 1 \}$.
  \end{enumerate}
\end{proposition}

The bounded sets in $\dot\cB'$ yield the seminorms on $\cD_{L^1}$ which already are known by definition. The seminorms on $\cD_{L^1,c}$ are described in \cite[Proposition 3.2]{BNO}, analogously to Proposition~\ref{prop5} \ref{prop5due}.

\begin{proof}
  \ref{prop10uno} The statement can be proven analogously to the proofs in Proposition~\ref{prop6} \ref{prop6uno}, Proposition~\ref{prop8} \ref{prop8uno} and Proposition~\ref{prop9} \ref{prop9uno}.

  \ref{prop10due} is a consequence of the regularity of $\dot\cB'$ (see \cite[Section 3]{BNO}).
\end{proof}

Analogously to Proposition~\ref{prop4} we characterise relatively compact subsets in $\dot\cB$:

\begin{proposition}\label{prop11}
  Let $C \subset \dot\cB$. Then
  \begin{enumerate}[label=(\roman*)]
  \item\label{prop11uno} $C$ is relatively compact if and only if $\exists g\in \dot\cB, \exists \varphi_1,\ldots,\varphi_m \in \cD$ with $C \subset g(\sum_{i=1}^{m}\varphi_i * B_{1,\infty})$ or, equivalently: $\exists g\in \dot\cB, \exists \varphi_1,\ldots,\varphi_m \in \cD: C \subset \sum_{i=1}^{m} \varphi_i * (g B_{1,\infty})$, $B_{1,\infty}$ defined as in Proposition~\ref{prop10}.
  \item\label{prop11due} $C$ relatively compact $\Leftrightarrow$ $\exists g \in \dot\cB$ $\exists B \subset \dot \cB$ bounded: $C \subset gB$.
  \item\label{prop11tre} The seminorms $p_{g,\varphi}$, $(g,\varphi) \in \dot\cB \times \cD$, defined by
    \[ \cD'_{L^1} \ni S \mapsto p_{g,\varphi}(S) = \norm{g(\varphi * S)}_1 \]
    generate the topology of $\cD'_{L^1,c}$. An equivalent description is given by the seminorms
    \[ \cD'_{L^1} \ni S \mapsto \norm{\varphi * (gS)}_1. \]
    In both cases, using $\varphi\in \cS$ instead of $\varphi\in \cD$ results in an equivalent system of seminorms.
  \end{enumerate}
\end{proposition}

\begin{proof}
  \ref{prop11due} ``$\Leftarrow$'': By Leibniz' formula we obtain, for all $\alpha \in \bN_0^n$:
  \[ \pd^\alpha C \subset \sum_{\beta \le \alpha} \binom{\alpha}{\beta} \pd^{\alpha-\beta} g \pd^\beta B, \]
  i.e., the boundedness and the uniform smallness of all $\pd^\alpha C$. The $\infty$-equi\-continuity of $\pd^\alpha C$ is implied by the inequality
  \[ \norm{\tau_h ( \pd^\alpha C) - \pd^\alpha C}_\infty \le \norm{\pd^{\alpha + l} C}_{\infty} \cdot \abso{h},\quad h \in \bR^n, l = (1, \dotsc, 1), \]
  and the boundedness of $\pd^{\alpha + l}C$ in $\cC_0$. Therefore, $\pd^\alpha C$ is relatively compact by the Arzel\`a-Ascoli Theorem.

  ``$\Rightarrow$'': By the compact factorisation property of $\dot\cB$ in \cite[Thm.~3.4.]{Voigt} or by \cite[Corollary, p.~610]{Craw} (whose hypothesis is fulfilled by \cite[Prop.~1.6.]{Voigt}) we even infer the existence of $g \in \dot\cB$ and a \emph{compact} set $B \subset \dot\cB$ such that $C \subset g B$.

  \ref{prop11uno} ``$\Rightarrow$'': If $C \subset \dot\cB$ is relatively compact then by \ref{prop11due} there exist $g \in \dot\cB$ and $B \subset \dot\cB$ bounded such that $C \subset g B$. By Proposition~\ref{prop8} \ref{prop8due} $\exists \varphi_1,\ldots,\varphi_m \in \cD$ such that $C \subset \sum_{i=1}^{m} g(\varphi_i * B_{1,\infty})$. Because $\sum_{i=1}^{m}g(\varphi_i * B_{1,\infty})$ is compact $\exists g_1 \in \dot\cB$,  $\psi_1,\ldots,\psi_m \in\cD$ such that $C \subset \sum_{i,j=1}^{m} g_1(\psi_j * (g ( \varphi_i * B_{1,\infty})))$, i.e. $C \subset \sum_{i=1}^{m} \psi_i * (\tilde g B_{1,\infty})$, $\tilde g \coleq C \cdot g$ with suitable $C>0$, which is the equivalent form of the characterisation of relatively compact sets in $\dot\cB$.

  \ref{prop11tre} If $C \subset \dot\cB$ is compact and $S \in \cD'_{L^1}$ then:
  \begin{align*}
    p_C(S) &\le \sum_{i=1}^{m} \sup_{f \in B_{1,\infty}} \abso{\langle g(\varphi_i * f), S \rangle } = \sum_{i=1}^{m} \sup_{f \in B_{1,\infty}} \abso{\langle f, \check \varphi_i * (gS) \rangle } = \sum_{i=1}^{m} \norm{\check \varphi_i * (gS)}_1.
  \end{align*}

  Hence, the topology on $\cD'_{L^1}$ generated by the seminorms $p_{g, \varphi}$ is finer than $\beta(\cD'_{L^1}, \dot\cB)$. It also is coarser because the mapping $\cD'_{L^1} \to L^1$, $S \mapsto \varphi * (gS)$ is continuous.

  The equivalent description in \ref{prop11tre} immediately results from the second characterisation of relatively compact sets in \ref{prop11uno}.
\end{proof}

A topology generating family of seminorms in $\dot\cB'_c$ is given in 

\begin{proposition}\label{prop12}
  The seminorms $p_{g, \varphi}$, $(g, \varphi) \in \dot\cB \times \cD$ or $(g, \varphi) \in \dot\cB \times \cS$, defined by $\dot\cB' \ni S \mapsto p_{g, \varphi}(S) = \norm{\varphi * (gS)}_\infty$, generate the topology of $\dot\cB'_c$.
\end{proposition}

\begin{proof}
  This is a direct consequence of Proposition~\ref{prop4} \ref{prop4tre} for $q=\infty$.
\end{proof}

In \cite[Proposition 2.5]{BNO} the topology of $\cD_{L^p, c}$ is described by ``function''-seminorms. Analogously we have:

\begin{proposition}\label{prop13}
  The topology of $\dot\cB_c$ is described by the seminorms $p_{m, \varphi}$, $(m, \varphi) \in \bN_0 \times \dot\cB$ defined by
  \[ \dot\cB \ni f \mapsto p_{m, \varphi}(f) = \norm{\varphi ( L_{-2m} * f)}_\infty. \]
\end{proposition}
\begin{proof}
  Using that the topology of $\dot\cB_c$ is given by the seminorms $p_C$, $C \subset \cD'_{L^1}$ compact, $p_C$ defined by:
  \[ \dot\cB \ni f \mapsto p_C(f) = \sup_{S \in C} \abso{\langle f, S \rangle } \]
  the claim follows from the compact regularity of $\cD_{L^1}'$ together with Lemma~\ref{lemma2} where $L^q$ has to be replaced by $\cC_0$.
\end{proof}

As a consequence of Proposition~\ref{prop3} \ref{prop3tre} and  \ref{prop3quatro} we obtain the following description of the topology of $\dot\cB'$ by ``function''-seminorms in 
\begin{proposition}\label{prop14}
  The topology of $\dot\cB'$ is generated by the seminorms $p_\varphi$, $\varphi \in \cD$ or $\varphi \in \cS$, defined by
  \[ \dot\cB' \ni S \mapsto p_\varphi(S) = \norm{\varphi * S}_\infty. \]
\end{proposition}

\begin{lemma}\label{cor1}If $C \subset \cC_0$ is compact then there exists $\varphi \in \dot\cB$ with $C \subset \varphi B_{1,\infty}$.
\end{lemma}
The \emph{proof} is completely analogous to the proof of Lemma~\ref{lemma2}.

Analogously to Proposition~\ref{prop7} we characterise compact sets in $\dot\cB'$ by

\begin{proposition}\label{prop15}
  Let $C \subset \dot\cB'$. Then the following assertions are equivalent:
  \begin{enumerate}[label=(\roman*)]
  \item\label{prop15uno} $C$ is relatively compact,
  \item\label{prop15due} $\exists (m, \varphi_1,\ldots,\varphi_M) \in \bN_0 \times \dot\cB^M$: $C \subset \sum_{j=1}^{M} \varphi_j ( L_{-2m} * B_{1,\infty})$, $B_{1,\infty} = \{f\in \cC_0\colon \|f\|_\infty \leq 1\}$,
  \item\label{prop15tre} $\exists (m, \varphi) \in \bN_0 \times \dot\cB$: $C \subset L_{-2m} * (\varphi B_{1,\infty})$,
  \item\label{prop15quattro} $\exists \varphi \in \dot\cB$ $\exists D \subset \dot\cB'$, $D$ bounded: $C \subset \varphi D$.
  \end{enumerate}
\end{proposition}

\begin{proof}
  \ref{prop15uno} $\Rightarrow$ \ref{prop15tre}:
  If $C \subset \dot\cB'$ is relatively compact then by \cite[Proposition 3.1]{BNO} there exists $m \in \bN_0$ such that $L_{2m} * C$ is compact in $\cC_0$. By Corollary~\ref{cor1} $\exists \varphi \in \dot\cB$ such that $L_{2m} * C \subset \varphi B_{1,\infty}$ and $C \subset L_{-2m} * (\varphi B_{1,\infty})$.
  
  \ref{prop15tre} $\Rightarrow$ \ref{prop15due}:
  Leibniz' rule implies
  \begin{align*}
    C \subset (1-\Delta_n)^m (\varphi B_{1,\infty}) &= \sum_{\abso{\alpha},\abso{\beta} \le 2m} c_{\alpha\beta} \pd^\alpha \varphi \pd^\beta B_{1,\infty}\\ & \subset \sum_{|\alpha|\leq 2m} \partial^\alpha \varphi (1-\Delta_n)^m B_{1,\infty}
  \end{align*}
  i.e.~\ref{prop15due} by~\cite[6.6 (b), p.~160]{St}.

  \ref{prop15due} $\Rightarrow$ \ref{prop15quattro}: $B_{1,\infty}$ is bounded in $\cC_0$ and thus $L_{-2m} * B_{1,\infty}$ is bounded in $\dot{\cB}'$.  Since $\{\varphi_1,\ldots,\varphi_M\}$ is a compact subset of $\dot\cB$ we may use the compact factorisation property of $\dot\cB$, see Proposition~1.6 in~\cite[p.~336]{Voigt}, to obtain $\psi\in\dot\cB$ and $\psi_1,\ldots,\psi_M\in\dot\cB$ with $\varphi_j = \psi \psi_j$. Hence
  \[
    C \subset \sum_{j=1}^{M} \varphi_j L_{-2m}*B_{1,q} = \psi \Big(\sum_{j=1}^{M} \psi_j L_{-2m}*B_{1,q}\Big) =: \psi D
  \]
  where $D\subset\cB'$ is a bounded set.

  \ref{prop15quattro} $\Rightarrow$ \ref{prop15uno}: This can be shown analogously to the proof of the corresponding implication of Proposition~\ref{prop7}, replacing all occurrences of $\cD'_{L^q}$ by $\dot\cB'$.
\end{proof}

\section{Bounded subsets in the Fr\'echet spaces $(L^{\lowercase{p}})_\infty$, $(\cC_0)_\infty = S^0_\infty$, $S^{\lowercase{m}}_\infty$, $\lowercase{m} \in \bN_0$, and topology generating seminorms in the (LB)-spaces $(L^{\lowercase{q}})_{-\infty}$, $(\cM^1)_{-\infty}$, $\cO_C^{\lowercase{m}}  = \cS^{\lowercase{m}}_{-\infty}$,  $\lowercase{m} \in \bN_0$.}\label{sec5}

The spaces $(\cC_0)_\infty = \cO^0_C$, $\cS^m_\infty$, $\cO_C^m = \cS^m_{-\infty}$ are defined in \cite[Ex. 12, p.~90 and Ex. 9, p.~173]{H}. The space $(\cM^1)_{-\infty} = (\cS^{\prime 0})_{-\infty}$ is the space of temperate integrable measures defined in \cite[p.~241]{Sch1}. $(L^1)_{-\infty}$ is the space of absolutely regular, temperate distributions investigated in \cite[pp.~8--34]{D}. It is the ``biggest'' space of temperate distributions $f\in \cS'\cap L^1_{\mathrm{loc}}$ with the property that for $\varphi\in \cS$ the evaluation $\langle \varphi, f\rangle$ is given by integration. The space $(L^q)_{-\infty}$ is an obvious generalisation thereof.

The bounded sets in the Fr\'echet spaces $(L^p)_{\infty}$, $1 \le p \le \infty$, and $(\cC_0)_\infty = \cS^0_\infty$ are characterised in Lemma~\ref{lemma1}:
\[ B \textrm{ bounded in }(L^p)_\infty\textrm{ or in }(\cC_0)_\infty \Longleftrightarrow \exists \varphi \in \cS: B \subset \varphi B_{1,p} \textrm{ or }B \subset \varphi B_{1,\infty}, \]
$B_{1,\infty} = \{ f \in \cC_0 : \norm{f}_\infty \le 1 \}$.

The next Proposition shows that finite differentiability commutes with the characterisation of bounded sets:

\begin{proposition}\label{prop16}
  Let $B$ be a bounded set in the Fr\'echet space $(\dot\cB^m)_\infty = \cS^m_{\infty}$, $m \in \bN_0$. Then
  \[
    \begin{aligned}
      & B \textrm{ bounded} & \Longleftrightarrow &\ \exists \varphi \in \cS\ \forall \alpha \in \bN_0^n, \abso{\alpha} \le m: \pd^\alpha B \subset \varphi B_{1,\infty}
    \end{aligned}
  \]
  with $B_{1,\infty} = \{ f \in \cC_0: \norm{f}_\infty \le 1 \}$. 
\end{proposition}

\begin{proof}
  If $B$ is a bounded subset of $(\dot\cB^m)_\infty = \cS^m_\infty$ then $\forall \alpha$, $\abso{\alpha} \le m$, $\pd^\alpha B$ is bounded in $\cS^0_\infty$. Thus, there exist functions $\varphi_\alpha \in \cS$ such that $\pd^\alpha B \subset \varphi_\alpha B_{1,\infty}$.
\end{proof}

\begin{proposition}\label{prop17}
  A projective description of the (LB)-topology of the spaces $L^q_{-\infty}$, $1 \le q \le \infty$, $(\cM^1)_{-\infty}$, $(\cC_0)_{-\infty} = \cS^0_{-\infty} = \cO^0_C$ is given by the seminorms $p_\varphi$, $\varphi \in \cS$, defined by
  \begin{align*}
   L^q_{-\infty} \ni f \mapsto p_\varphi(f) &= \norm{\varphi f}_q,\\
   \cS^0_{-\infty} \ni f \mapsto p_\varphi(f) &= \norm{\varphi f}_\infty\\
   \intertext{and}
  \cM^1_{-\infty} \ni \mu \mapsto p_\varphi(\mu) &= \norm{\varphi \mu}_1,
  \end{align*}
  respectively.
\end{proposition}

\begin{proof}
  Let $1 \le p < \infty$ and $B$ be a bounded subset of $L^p_\infty$ or $\cS^0_\infty$. Then the seminorms $p_B$, defined by $p_B(f) = \sup_{g \in B} \abso{\langle g, f \rangle}$, $f \in (L^q)_{-\infty}$, $f \in \cO^0_C$, or $p_B(\mu) = \sup_{g \in B} \abso{\langle g, \mu \rangle }$, $\mu \in (\cM^1)_{-\infty}$, generate the topology of $(L^q)_{-\infty}$, $\cO_C^0$ or $(\cM^1)_{-\infty}$, respectively. We obtain for $f \in (L^q)_{-\infty}$, $\cO^0_C$ or $\mu \in (\cM^1)_{-\infty}$:
  \[ p_B(f) \le \sup_{h \in B_{1,p}} \abso{\langle \varphi h, f \rangle} = \norm{\varphi f}_q \]
  and
  \[ p_B(\mu) \le \norm{\varphi \mu}_1. \]
  The topology of $(L^1)_{-\infty}$ is the induced topology of $(\cM^1)_{-\infty}$. The topology of $(\cC_0)_{-\infty} = \cS^0_{-\infty} = \cO^0_C$ is the induced topology of $(L^\infty)_{-\infty}$.
\end{proof}

\begin{remarks}
  \begin{enumerate}[label=(\arabic*)]
  \item The projective description of the (LB)-topology of $(L^1)_{-\infty}$ is given for the first time in \cite[3.1, Satz, p.~19 and p.~25]{D}.
  \item The projective description of the topology of $\cS^0_{-\infty} = \cO^0_C$ is given in \cite[Prop. 1 and Prop.~2]{OW}.
  \end{enumerate}
\end{remarks}

\begin{proposition}\label{prop18}
  A projective description of the (LB)-topology of the space $\cO_C^m = \cS^m_{-\infty}$, $m \in \bN_0$, is given by the seminorms $p_\varphi$, $\varphi \in \cS$, defined by:
  \[ \cO_C^m \ni f \mapsto p_\varphi(f) = \sup_{\abso{\alpha} \le m } \norm{\varphi \pd^\alpha f}_{\infty}. \]
\end{proposition}

\begin{proof}
  It follows from the projective limit representation
  \[ \cO_C^m = \varprojlim_{\abso{\alpha} \le m} (\pd^\alpha)^{-1} \cO^0_C\]
  and Proposition~\ref{prop17}.
\end{proof}

\begin{remark}
  A more direct proof of Proposition~\ref{prop18} is given in \cite[Proposition 2]{OW}.
\end{remark}

\section{Bounded sets in the Fr\'echet spaces $L^{\lowercase{q}}_{\textrm{\lowercase{loc}}}$, $\cE^{\lowercase{m}}$, $\cM$, $\cD^{\prime {\lowercase{m}}}$, $\lowercase{m} \in \bN_0$, and topology generating seminorms in the strict (LB)-spaces $L^{\lowercase{p}}_{\lowercase{comp}}$, $\cM_{\lowercase{comp}}$, $\cD^0 = \cK$ and $\cD^{\lowercase{m}}$, $\lowercase{m} \in \bN_0$.}\label{sec6}

The function and distribution spaces considered in this section, i.e., $L^p_\textrm{comp}(\Omega)$, $L^q_\textrm{loc}(\Omega)$, $\cE^m(\Omega)$, $\cD^m(\Omega)$ and $\cD^{\prime m}(\Omega)$, $m \in \bN_0$, are defined on an open subset $\Omega$ of $\bR^n$. Generally, we shall omit the set $\Omega$.

Obviously, the statements in this section on the spaces $L^p_\textrm{comp}(\Omega)$, $L^q_\textrm{loc}(\Omega)$, $\cE^0(\Omega)$, $\cE^{\prime 0}(\Omega) = \cM_\textrm{comp}(\Omega)$, $\cD^0(\Omega) = \cK(\Omega)$ and $\cD^{\prime 0}(\Omega) = \cM(\Omega)$ remain valid if the set $\Omega$ is replaced by a locally compact space $X$.

The spaces $\cE^m$, \cite[Ex. 7, p.~364]{H}, $\cD^m$, $\cD^{\prime m}$ are treated in textbooks on distribution theory. $\cD^0 = \cK$ is the space of test functions whose dual defines the Radon measures (see \cite[p.~96]{Dieu2}; \cite[Chap. II, D\'ef.~2, p.~47]{BourbakiInt}), $L^1_\textrm{loc}$ is the ``biggest space of functions'' which furnishes ``distributions by integration''. The spaces $L^p_\textrm{loc}$ and $L^q_\textrm{comp}$ considered in~\cite{Hor} are special cases of local spaces of distributions treated in \cite{HorvathLocal}.

First we characterise bounded sets in the Fr\'echet spaces $L^q_\textrm{loc}$, $\cE^0$, $\cM$ by functions $\varphi \in \cE$ in contrast to the characterisations of bounded sets in $(L^p)_{\infty}$, $(\cC_0)_\infty = \cS^0_\infty$, $\cM^1_{\infty}$ (in Lemma~\ref{lemma1}) by functions $\varphi \in \cS$.

\begin{lemma}\label{lemma4}
  Let $\Omega \subseteq \bR^n$ be open and $B \subseteq L^q_\textrm{loc}(\Omega)$, $1 \le q \le \infty$, or $B \subset \cE^0(\Omega)$ or $B \subset \cM(\Omega)$. Then $B$ is bounded if and only if there is $\varphi \in \cE_+(\Omega)$ such that $B \subset \varphi B_{1,q}$ or $B \subset \varphi B_{1,\infty}$ or $B \subset \varphi B_{1,1}$, respectively. Here, $B_{1,q} = \{ f \in L^q(\Omega): \norm{f}_q \le 1 \}$, $B_{1,\infty} = \{ f \in \cC_0 : \norm{f}_\infty \le 1 \}$ or $B_{1,1} = \{ \mu \in \cM^1(\Omega) : \norm{\mu} \le 1\}$.
\end{lemma}

\begin{proof}
  We show the characterisation for $L^q_{\mathrm{loc}}(\Omega)$. The proof for $\cE^0(\Omega)$ can be obtained by replacing all occurrences of the $L^q$-norm with the $L^\infty$-norm. For the case of $\cM(\Omega)$ we have to take into account the representation as the projective limit $\cM(\Omega)=\varprojlim (\cD^0_K(\Omega))'$ which follows from the fact that $\cK(\Omega)=\varinjlim \cD^0_K(\Omega)$ is a strict inductive limit, see~\cite{Sch5} and Proposition~1 in~\cite{B}. For $\cM(\Omega)$ the characteristic function $\chi_K$ has to be replaced by a continuous function which is one on $K$ and vanishes outside a neighbourhood of $K$.
  
  ``$\Leftarrow$'': obvious in virtue of $\norm{\chi_K f }_q \le \norm{\chi_K \varphi}_{\infty}$ for all $f\in B$ if $K \subset \Omega$ is a compact subset and $\chi_K$ its characteristic function. 

  ``$\Rightarrow$'': Let $(\Omega_k)_{k \in \bN}$ be a locally finite open cover of $\Omega$ by relatively compact subsets $\Omega_k$. Then, there exist functions $\varphi_k \in \cD(\Omega)$, $\varphi_k(x) \ge 0$ for all $x \in \Omega$ such that $\supp \varphi_k \subset \Omega_k$, the family $(\supp \varphi_k)_{k \in \bN}$ is locally finite and $\sum_{k=1}^\infty \varphi_k(x) = 1$ if $x \in \Omega$ \cite[Theorem 4, p.~168]{H}. By defining $a(k) = 1 + \sup_{f \in B} \norm{\varphi_k \cdot f}_q$ we see that the function
  \[ \bR^n \to \bR, \quad x \mapsto \sum_{k=1}^\infty 2^{-k} \frac{\varphi_k(x)}{a(k)} \]
  is well-defined, positive and belongs to $C^\infty(\Omega$). Hence,
  \[ \varphi \coleq \left( \sum_{k=1}^\infty 2^{-k} \frac{\varphi_k}{a(k)}\right)^{-1} \in \cE_+(\Omega) \]
  and
  \[ \norm{ \frac{1}{\varphi} f }_q = \norm{\sum_{k=1}^\infty 2^{-k} \frac{\varphi_k f}{a(k)}}_q \le \sum_{k=1}^\infty 2^{-k} \frac{\norm{\varphi_k f}_q}{a(k)} \le \sum_{k=1}^\infty 2^{-k} = 1 \]
  if $f \in B$. Thus, $B \subset \varphi B_{1,q}$.
\end{proof}

\begin{proposition}\label{prop20}
  The strict (LB)-topology of the spaces $(L^p_\textrm{comp}, \beta(L^p_\textrm{comp}, L^q_\textrm{loc}))$, $1 \le p \leq \infty$, $1/p + 1/q = 1$, $(\cM_\textrm{comp}, \beta(\cM_\textrm{comp}, \cE^0))$, $\cM_\textrm{comp} \cong \cE^{\prime 0}$, $(\cK, \beta(\cK, \cM))$, $\cK = \cD^0$, is generated by the seminorms $p_\varphi$, $\varphi \in \cE_+$,
  \begin{align*}
    L^p_\textrm{comp} \ni f &\mapsto p_\varphi(f) = \norm{\varphi f}_p, \\
    \cM_\textrm{comp} \ni \mu & \mapsto p_\varphi(\mu) = \norm{\varphi \mu}_1, \\
    \cK \ni f &\mapsto p_\varphi(f) = \norm{\varphi f}_\infty,
  \end{align*}
  respectively.

  These are projective descriptions of the strict (LB)-topologies. 
\end{proposition}

\begin{proof}
  We only treat the case of seminorms in $L^p_\textrm{comp}$, $1<p\leq\infty$. If $B$ is a bounded set in $L^q_\textrm{loc}$ then the seminorms $p_B$, defined by
  \[ L^p_\textrm{comp} \ni f \mapsto p_B(f) = \sup_{g \in B} \abso{\langle f, g \rangle}, \]
  generate the topology of $L^p_\textrm{comp}$. There exists $\varphi \in \cE_+$ such that $B \subset \varphi B_{1,q}$ by Lemma~\ref{lemma4}. Hence,
  \[ p_B(f) \le \sup_{h \in B_{1,q}} \abso{ {}_{L^q} \langle h, \varphi f \rangle_{L^p}} = \norm{\varphi f}_p. \]
  Conversely, the continuity of the mapping $L^p_\textrm{comp} \to L^p$, $f \mapsto \varphi f$ furnishes the equivalence of $\beta(L^p_\textrm{comp}, L^q_\textrm{loc})$ and the topology generated by $p_\varphi$, $\varphi \in \cE_+$.
\end{proof}

We show in the next Proposition that the (LB)-topologies of the spaces $L^p_\textrm{comp}$, $\cM_\textrm{comp}$ and $\cK$ also can be identified with weak and strong topologies of operator spaces in which they can be embedded by multiplication:

\begin{proposition}\label{prop21}
  The topologies of $L^p_\textrm{comp}$, $\cM_\textrm{comp}$ and $\cK$ are the induced topologies of the spaces $\cL_s(\cE, L^p)$ or $\cL_b(\cE, L^p)$, $\cL_s(\cE, \cM^1)$ or $\cL_b(\cE, \cM^1)$ and $\cL_s(\cE, \cB^0)$ or $\cL_b(\cE, \cB^0)$, respectively, if they are ``embedded by multiplication'', i.e., $L^p_\textrm{comp} \to \cL(\cE, L^p)$, $f \mapsto [\varphi \mapsto \varphi f]$, etc.
\end{proposition}

\begin{proof}
  We restrict the proof to the space $L^p_\textrm{comp}$. The mapping
  \[ L^p_\textrm{comp} \to \cL(\cE, L^p), f \mapsto  [\varphi \mapsto \varphi f] \]
  is injective, hence it is an embedding. The identity of the topology of $L^p_\textrm{comp}$ with the induced topology is a consequence of the partial continuity on the one hand and of the hypocontinuity on the other hand of the bilinear multiplication mapping
  \[ L^p_\textrm{comp} \times \cE \to L^p,\quad (f, \varphi) \mapsto f \varphi. \qedhere \]
\end{proof}

Next, let us treat the case of finite differentiability.

\begin{proposition}\label{prop22}
  The strict (LB)-topologies of the spaces $\cD^m$, $m \in \bN_0$, can be described projectively by the seminorms $p_\varphi$, $\varphi \in \cE$, defined by
  \[ \cD^m \ni f \mapsto p_\varphi(f) = \sup_{\abso{\alpha} \le m } \norm{\varphi \pd^\alpha f}_{\infty }. \]
\end{proposition}

The \emph{proof} is a direct consequence of the representation as a finite projective limit, i.e. of $\cD^m = \varprojlim_{\abso{\alpha} \le m} ( (\pd^\alpha)^{-1} \cD^0)$, $\cD^0=\cK$,
and Proposition~\ref{prop20}.

\begin{remark}
  Note that the topology on $\cD$ defined by the seminorms $p_{m,\varphi}$, $m \in \bN_0$, $\varphi \in \cE$,
  \[ \cD \ni f \mapsto p_{m, \varphi}(f) = \sup_{\abso{\alpha} \le m} \norm{\varphi \pd^\alpha f}_\infty, \]
  is the topology of the space $\cD^F = \varprojlim_{m \in \bN_0} \cD^m$, which is coarser than the topology of $\cD$, $\cD = \varinjlim_{K\textrm{ compact}} \cD_K$ (see \cite[p.99]{Sch5}). L.~Schwartz remarks there that the topology of $\cD_F$ on $\cD$ is the ``natural'' topology. That the topologies of $\cD_F$ and $\cD$ are different can be easily seen by the known example: $S = \sum_{m=0}^\infty \delta_m^{(m)} \in \cD'(\bR)$ but $S \not\in \cD^{\prime F}(\bR)$, the order of $S$ is not finite. For a more detailed discussion of $\cD^F$ and a sequence space representation see~\cite{BargetzDF}.
\end{remark}

Finally, we describe the bounded sets in $\cE^m$ and in $\cD^{\prime m}$, $m \in \bN_0$:

\begin{proposition}
  The bounded subsets of the space $\cD^{\prime m}$ can be characterised by
  \[ \exists \varphi \in \cE: B \subset \pd^\alpha(\varphi B_{1,1}),\ \abso{\alpha} \le m,\]
  where $B_{1,1} = \{ \mu \in \cM^1: \norm{\mu} \le 1 \}$. The bounded subsets $B$ of the Fr\'echet space $\cE^m$ can be characterised by:
  \[ \exists \varphi \in \cE,  \pd^\alpha B \subset \varphi B_{1,\infty} \forall  \abso{\alpha} \le m. \]
\end{proposition}

The \emph{proof} results from the representations as finite projective limits:
\[ \cD^{\prime m} = \varprojlim_{\abso{\alpha} \le m} \pd^\alpha \cM, \qquad \cE^m = \varprojlim_{\abso{\alpha} \le m} (\pd^\alpha)^{-1} \cE^0, \]
respectively, and the description of bounded sets in $\cM$ and $\cE^0$ in Lemma~\ref{lemma4}.

\section{Projective description of Mackey topologies and Buck's topology on the space $\cB^{\lowercase{m}}$, $\lowercase{m} \in \bN_0$.}\label{sec7}

In \cite[Proposition 6.4]{BNO}, a characterisation of relatively compact subsets $C$ of $\cM^1$ is given: $\exists (g,h) \in \cC_0 \times L^1: C \subset g(h * B_{1,1})$, $B_{1,1} = \{ \mu \in \cM^1: \norm{\mu}_1 \le 1 \}$. This is not correct: A counterexample is $C = \{ \delta \}$. We prove the following modification, essentially due to A.~Grothendieck:

\begin{proposition}\label{prop24}
  For a subset $C$ of $\cM^1$ the following assertions are equivalent:
  \begin{enumerate}[label=(\roman*)]
  \item\label{prop24uno} $C$ is $\sigma(\cM^1, (\cM^1)')$-relatively compact;
  \item\label{prop24due} $\exists g \in \cC_0$: $C \subset g B_{1,1}$, $B_{1,1} = \{ \mu \in \cM^1: \norm{\mu}_1 \le 1 \}$;
  \item\label{prop24tres} $C$ is $\tau(\cM^1, \cC_0)$-relatively compact.
  \end{enumerate}
\end{proposition}

\begin{proof}
  \ref{prop24uno} $\Rightarrow$ \ref{prop24due}: the relative compactness of $C$ with respect to the topology $\sigma(\cM^1, (\cM^1)')$ implies the existence of $g \in \cC_0$ such that $C \subset g B_{1,1}$: apply \cite[Theorem 2, (4) (b), pp.~146--147]{GrothenCanadian}; compare also \cite[p.~413]{GrothendieckEVTII}.

  \ref{prop24due} $\Rightarrow$ \ref{prop24uno}: Let $(\varphi_k)_k$ be a sequence converging to $0$ in $\sigma(\cC_0, \cM^1)$. By Theorem~2 in~\cite{GrothenCanadian} in order to show the relative compactness of $C$ we have to prove $\langle \varphi_k, g B_{1,1} \rangle = \langle g \varphi_k, B_{1,1} \rangle \to 0$ if $k \to \infty$. The continuity of the multiplication $\cC_0 \times \cM^1 \to \cM^1$ implies its weak partial continuity, i.e., $(\cC_0, \sigma(\cC_0, \cM^1)) \times \cM^1 \to (\cM^1, \sigma(\cM^1,(\cM^1)'))$. Thus, the linear forms $T_k$, defined by $T_k \colon \cM^1 \to \bC$, $\mu \mapsto {}_{\cC_0} \langle g \varphi_k, \mu \rangle_{\cM^1}={}_{\cM^1} \langle \mu \varphi_k,  g \rangle_{\cC_0}$ converge on $\cM^1$ pointwise. $\cM^1$ is a Banach space such that $\langle g \varphi_k, B_{1,1} \rangle \to 0$ if $k \to \infty$ by the Theorem of Banach-Steinhaus.

  \ref{prop24uno} $\Longleftrightarrow$ \ref{prop24tres}: see \cite[Corollaire du Th\'eor\`eme 2, p.~149]{GrothenCanadian}.
\end{proof}

\begin{remarks}
  \begin{enumerate}[label=(\arabic*)]
  \item Concerning Buck's topology on the space $\cB^0 = \cB\cC$: If $f \in \cB^0$ and $g \in \cC_0$ then $\norm{g f }_\infty = \sup_{\norm{\mu} \le 1 } \abso{\langle f, g \mu \rangle }$,
    which shows that Buck's topology on $\cB^0$ is the topology of uniform convergence on $\tau(\cM^1, \cC_0)$-relatively compact subsets of $\cM^1$. If we denote (as in \cite[Prop. 1.2.1, p.~6]{DVAF}) Buck's topology on $\cB^0$ by an index $b$ (not to be confounded with ``b'' in $\cL_b(E,F))$ then we have:
    \[ \cB^0_b = (\cB^0, \kappa(\cB^0, (\cM^1, \tau(\cM^1, \cC_0)))). \]

  \item Concerning a generalisation of Buck's topology on the space $\cB = \cD_{L^\infty}$: The description of the topology $\kappa(\cB, \cD'_{L^1}) = \tau ( \cB, \cD'_{L^1})$ by the seminorms $p_{g, \alpha}$, $g \in \cC_0$, $\alpha \in \bN_0^n$, defined by $\cB \ni \varphi \mapsto p_{g, \alpha}(\varphi) = \norm{ g \pd^\alpha \varphi}_\infty$, is given in \cite[p.~51 and (3.5) Cor., p.~71]{DD}.
  \end{enumerate}
\end{remarks}

An application of Proposition~\ref{prop24} is the description of Buck's topology on the space $\cB^m$, $m \in \bN_0$:

\begin{proposition}\label{prop25}
  Let $m \in \bN_0$.
  \begin{enumerate}[label=(\roman*)]
  \item\label{prop25uno} The $\tau(\cD^{\prime m}_{L^1}, \dot\cB^m)$-relatively compact subsets $C$ of $\cD^{\prime m}_{L^1}$ are characterised by:
    \[ \forall \alpha, \abso{\alpha} \le m, \exists g_\alpha \in \cC_0: C \subset \pd^\alpha(g_\alpha B_{1,1}). \]
  \item\label{prop25due} Seminorms which generate the topology $\kappa(\cB^m, (\cD^{\prime m}_{L^1}, \tau(\cD^{\prime m}_{L^1}, \dot\cB^m)))$ are $p_g$, $g \in \cC_0$, defined by
    \[ \cB^m \ni \varphi \mapsto p_g(\varphi) = \sup_{\abso{\alpha} \le m} \norm{g \pd^\alpha \varphi}_\infty. \]
  \end{enumerate}
\end{proposition}

\begin{proof}
  \ref{prop25uno} The characterisation of relatively compact subsets $C$ of $\cD^{\prime m}_{L^1}$ follows by the representation of $\cD^{\prime m}_{L^1}$ as a finite inductive limit and the corresponding characterisation of subsets in $\cM^1$ (Proposition~\ref{prop24}).
  
  \ref{prop25due} follows from \ref{prop25uno}.
\end{proof}

Finally, we describe the (LB)-topology on J. Horvath's space $\cO^m_C$, $m \in \bN_0$, by seminorms.

\begin{proposition}\label{prop26}
  For $m\in\mathbb{N}_0$ a subset $B$ of the Fr\'{e}chet space $\cO_C'^m=\cS'^m_\infty= (\cS_{-\infty}^m)' = (\cO_C^m)'$ is bounded if and only if there is a $\varphi\in\cS$ such that for all $\alpha\in\mathbb{N}_0^n$ with $|\alpha|\leq m$ we have
  \[
    B \subset \partial^\alpha (\varphi B_{1,1})
  \]
  where $B_{1,1} = \{\mu\in\cM^1\colon \|\mu\|_1 \leq 1 \}$.
\end{proposition}

\begin{proof}
  First note that $B$ is bounded if and only if for every $\alpha\in\mathbb{N}_0^n$ with $|\alpha|\leq m$ there is a $B_\alpha\subset \cO_C'^{0}=(\cS^0_{-\infty})' = (\cM^1)_\infty$ with $B=\partial^\alpha B_\alpha$. This means that for every $k\in\mathbb{N}_0$ the set $(1+|x|^2)^{k} \bigcup_{|\alpha|\leq m}B_\alpha$ is bounded in $\cM^1$ which by Lemma~\ref{lemma1} is equivalent to the existence of a function $\varphi\in\cS$ with the property that $\bigcup_{|\alpha|\leq m}B_\alpha \subset\varphi B_{1,1}$. This finishes the proof of the claimed characterisation.
\end{proof}

\begin{remark}
  The characterisation of bounded subsets of $\cO^{\prime m}_C$ in Proposition~\ref{prop26} again shows that the seminorms
  \[
    \cO^m_C \ni f \mapsto p_{m,g}(f) = \sup_{|\alpha|\leq m} \|g\partial^\alpha f\|_\infty, \qquad \varphi\in \cS
  \]
  generate the topology of $\cO^m_C$. Hence it provides a third, simpler proof of \cite[Proposition 2]{OW}.
\end{remark}

\textbf{Acknowledgement.} The authors would like to thank Andreas Debrouwere and Jasson Vindas for valuable comments. The authors also thank the anonymous referees for their remarks.

\end{document}